\documentclass[11pt,a4paper]{amsart}
\usepackage{amsmath,amsfonts,xargs,amssymb}

\usepackage[utf8]{inputenc}
\usepackage{graphicx} 
\usepackage{amsthm}
\usepackage{tikz}
\usetikzlibrary{calc}
\usepackage{tkz-tab}
\usepackage{theoremref}
\usepackage{thmtools, thm-restate}
\declaretheorem{theorem}
\usepackage{hyperref}
\usepackage{comment}
\usepackage{color}
\definecolor{RED}{rgb}{1,0,0} 
\usepackage{mathtools}
\usepackage{subcaption}
\mathtoolsset{showonlyrefs}
\hypersetup{colorlinks=true,urlcolor=black, pdftitle={Volume and Projection Inequalities II: Determinants and Lp-Sums}}

\usepackage[T1]{fontenc}
\usepackage{lmodern}
\usepackage{microtype}
\usepackage{booktabs, array,longtable}
\usepackage{enumitem}

\theoremstyle{plain}

\newtheorem{lemma}[theorem]{Lemma}

\newtheorem{prop}[theorem]{Proposition}

\newtheorem{conj}{Conjecture}
\newtheorem{ques}[conj]{Question}

\theoremstyle{definition}

\newcommand{\MainTheoremExtra}{.}

\newcommand{\R}{\mathbb R}
\newcommand{\s}{\mathbb S}

\title[Volume and Projection Inequalities II]{Volume and Projection Inequalities II:\\ Determinants and $L_p$-Sums}
\author[M. Fradelizi, A. Manui, C. S. Ndiaye, and A. Zvavitch]{Matthieu Fradelizi,  Auttawich Manui,
Cheikh Saliou Ndiaye, and Artem Zvavitch}

\date{}

\begin{document}

\begin{abstract}
We study  inequalities  for the volume of orthogonal projections and their relation to  Firey 
$L_p$-sum, together with their determinant-power analogues, motivated
by the Dembo--Cover--Thomas conjecture. For $L_p$-zonoids $K,L\subset\mathbb{R}^n$ and $u\in S^{n-1}$,
we consider the inequality
\[
\left(
\frac{|K\oplus_p L|}
     {|P_{u^\perp}(K\oplus_p L)|}
\right)^p
\geq
\left(
\frac{|K|}{|P_{u^\perp}K|}
\right)^p
+
\left(
\frac{|L|}{|P_{u^\perp}L|}
\right)^p .
\]
For every $1<p<2$, we prove that this inequality fails in every
dimension $n\geq2$. In contrast, the weak one-term inequality,
obtained by omitting the second term on the right-hand side, holds in
dimension two throughout the full range $1\leq p\leq2$. The proof of
this planar result uses a sharp estimate for the normalized
duality map.

We also classify the corresponding determinant-power inequalities in
the range $0<p<2$. The strong two-term inequality holds in dimension
two and fails in every dimension $n\geq3$. The weak one-term
inequality holds for $0<p\leq1$ in dimensions $n\leq3$ and fails
for $n\geq4$; for $1<p<2$, it holds only in dimension two.
\end{abstract}

\maketitle
\tableofcontents

\section{Introduction}

The present paper continues a line of research on inequalities for the volume of convex bodies and their orthogonal projections. Those inequalities are largely 
 inspired by connections with additive combinatorics and
information theory. Among their main motivations is a conjecture of
Dembo, Cover, and Thomas \cite{DCT-91}. They asked whether, for convex
bodies \(A,B\subset\mathbb{R}^n\),
\begin{equation}\label{conj:DCT}
\frac{|A+B|}{|\partial(A+B)|_{n-1}}
\geq
\frac{|A|}{|\partial A|_{n-1}}
+
\frac{|B|}{|\partial B|_{n-1}}.
\end{equation}
Here, \(|\cdot|\) denotes \(n\)-dimensional volume and
\(|\partial K|_{n-1}\) denotes the surface area of \(K\).

Recall that a zonotope is a finite Minkowski sum of symmetric
segments, and a zonoid is a Hausdorff limit of zonotopes. Within the
class of zonoids, the validity of the inequality above is equivalent
to
\begin{equation}\label{conj:str-zonoid}
\frac{|A+B|}{|P_{u^\perp}(A+B)|}
\geq
\frac{|A|}{|P_{u^\perp}A|}
+
\frac{|B|}{|P_{u^\perp}B|}
\end{equation}
for every \(u\in S^{n-1}\), where \(P_{u^\perp}A\) denotes the
orthogonal projection of \(A\) onto \(u^\perp\), see \cite[Remark~3.10]{FMMZ-24}.

The strong projection inequality above, together with its weak
one-term version, is studied in our companion paper
\cite{FHMNWZ-26-I}. The present paper considers two natural extensions:
one based on Firey's \(L_p\)-addition and the other on sums of powers
of determinants.
\subsection{Inequalities for volume.}
The first extension replaces Minkowski addition by Firey
\(L_p\)-addition, introduced by Firey \cite{F-62} and subsequently
developed systematically by Lutwak into what is now the rich and
far-reaching \(L_p\) Brunn--Minkowski theory
\cite{L-93,L-96}. Let
$p\geq1$ 
and $K,L\subset\R^n$ be convex bodies containing the origin, the
$L_p$-sum of $K$ and $L$ is the convex body $K\oplus_p L$ defined by
\[
 h_{K\oplus_p L}^p(x)=h_K^p(x)+h_L^p(x),
 \qquad x\in\R^n,
\]
where $h_K(u)=\max_{x\in K}\langle x, u \rangle$ denotes the support function of $K$. 

An $L_p$-zonotope is a finite $L_p$-sum of symmetric segments. Hausdorff limits of $L_p$-zonotopes are called $L_p$-zonoids.  When $p=1$, these are zonoids, whereas for $p=2$, they are ellipsoids.

The strong $L_p$ projection problem \cite{FMMZ-24} asks whether the $p$th power of the
volume-to-projection ratio is superadditive.

\begin{ques}\label{ques:ratio-Lp-zonoid}
Let $p\geq1$ and $n \geq 2$. Is it true that for all $L_p$-zonoids $A,B\subset\R^n$ and every
$u\in\s^{n-1}$,
\begin{equation}\label{eq:question1}
 \left(\frac{|A\oplus_p B|}{|P_{u^\perp}(A\oplus_p B)|}\right)^p
 \geq
 \left(\frac{|A|}{|P_{u^\perp}A|}\right)^p
 +\left(\frac{|B|}{|P_{u^\perp}B|}\right)^p?
\end{equation}
\end{ques}

When $p=1$, the inequality reduces to \eqref{conj:str-zonoid}
for zonoids. The planar case follows from Bonnesen's inequality
\cite{B-1929}, whereas in \cite{FHMNWZ-26-I} we show that it fails in
every dimension $n\geq 3$.  For $p=2$, the answer is affirmative in every dimension, whereas the answer is negative
for $p>2$; see~\cite[Theorem~6.2 and Proposition~6.8]{FMMZ-24}.  Our first result settles the remaining range $p\in (1,2)$.

\begin{restatable}{theorem}{RatioLpZonoid}\label{thm:ratio-Lp-zonoid}
For every $1<p<2$ and every $n\geq2$, there exist $L_p$-zonotopes
$A,B\subset\R^n$ and $u\in\s^{n-1}$ for which
\eqref{eq:question1} fails\MainTheoremExtra
\end{restatable}

Motivated by the log-submodularity property of volumes studied in~\cite{FHMNWZ-26-I}, we also consider the weaker one-term inequality.

\begin{ques}\label{ques:ratio-Lp-zonoid_weak}
Let $p\geq1$, and $n \geq 2$. Is it true that for all $L_p$-zonoids $A,B\subset\R^n$ and every
$u\in\s^{n-1}$,
\begin{equation}\label{eq:question1weak}
 \frac{|A\oplus_p B|}{|P_{u^\perp}(A\oplus_p B)|}
 \geq
 \frac{|A|}{|P_{u^\perp}A|}?
\end{equation}
\end{ques}
When $p =2$, this weaker inequality follows immediately from the strong inequality and hence it holds in every dimension. For $p>2$, the inequality was shown in \cite{FMMZ-24} to fail in
every dimension. For $p=1$, it holds in dimensions $n\leq 3$
\cite{B-1929,FMMZ-24}, whereas our companion paper
\cite{FHMNWZ-26-I}  and \cite{S-26} show that it fails in every dimension $n\geq 4$.  Continuity of Firey's sum in $p$ consequently gives counterexamples in dimensions $n\geq4$ for $p>1$ sufficiently close to $1$. Unlike the strong inequality, we show that the planar weak inequality holds throughout the full interval $p \in [1,2]$.

\begin{restatable}{theorem}{RatioLpZonoidweak}\label{thm:ratio-Lp-zonoidwekk}
Let $1\leq p\leq2$, let $A,B\subset\R^2$ be $L_p$-zonoids, and let
$u\in\s^1$. Then \eqref{eq:question1weak} holds.
\end{restatable}

The higher-dimensional weak problem remains open in the intermediate range.
We conjecture that it is affirmative for $1<p<2$ in dimension three and
negative for every $1<p<2$ in dimensions $n\geq4$.

Question \ref{ques:ratio-Lp-zonoid} is closely related to determinant inequalities. For $p =1$ and $p =2$, the relevant volume expressions admit a direct representation of the volume of $L_p$-zonotopes in terms of determinants, see \cite{FMMZ-24, GJ-97}. This leads naturally to the following determinant question proposed in \cite{FMMZ-24}.

\subsection{Inequalities for determinant}
Let
$\mathcal U=(u_1,\ldots,u_N)$ be a finite sequence in $\R^n$, let
$v\in\s^{n-1}$, and, for $1\leq a\leq b\leq N$, set
\begin{equation}\label{eq:determinant-ratio-intro}
 R_{\mathcal U,v}^{(p)}[a,b]
 =
 \frac{\displaystyle
   \sum_{\substack{M\subset\{a,\ldots,b\}\\|M|=n}}
   |\det(u_m)_{m\in M}|^p}
 {\displaystyle
   \sum_{\substack{M\subset\{a,\ldots,b\}\\|M|=n-1}}
   |\det(P_{v^\perp}u_m)_{m\in M}|^p}.
\end{equation}
An empty sum is understood as zero, and the ratio is defined to be zero when
its denominator vanishes. We also note that the projected determinant is computed in any orthonormal basis of $v^{\perp}$, equivalently as the $(n-1)$-dimensional volume of the projected parallelotope.

The sums of powers of minors appearing in this definition are closely
related to the vector-valued Maclaurin inequalities studied by
Brazitikos and McIntyre~\cite{BM-22}. Their inequalities compare
normalized sums of the $k$ and $(k-1)$ dimensional volumes
generated by a single family of vectors. In our setting, the numerator
is the top-dimensional sum, whereas averaging the denominator over
$v\in {\mathbb S}^{n-1}$ yields, up to a constant depending only on $n$ and
$p$, the corresponding $(n-1)$-dimensional sum. Questions \ref{ques:p-determinant}
and \ref{ques:p-determinant-weak} below are directional analogues concerned, respectively, with
splitting and enlarging the family of vectors.

\begin{ques}\label{ques:p-determinant}
For $p>0$ and $N'< N$, is it true that
\begin{equation}\label{eq:p-determinant}
 R_{\mathcal U,v}^{(p)}[1,N] \geq R_{\mathcal U,v}^{(p)}[1,N']
 +R_{\mathcal U,v}^{(p)}[N'+1,N]?
\end{equation}
\end{ques}
When $p=1$, the numerator and denominator are, respectively, the volume of the zonotope generated by the $u_i$ and the volume of its projection onto $v^\perp$. Consequently, at $p=1$, the
question below is precisely the strong projection inequality for zonotopes,
which is completely answered as explained above.  The determinant-power extension was posed
in~\cite[Question~6.10]{FMMZ-24}, where a counterexample was also given for $p >2$.
For $p=2$, the inequality was proved in every dimension as part of \cite[Theorem 6.2]{FMMZ-24}. The result below provides a complete answer to the remaining cases in this range.

\begin{restatable}{theorem}{DetIneq}\label{thm:determinant-ineq-2}
Let $0<p<2$.  The answer to Question~\ref{ques:p-determinant} is affirmative
in dimension $n=2$ and negative in every dimension $n\geq3$.
\end{restatable}

Removing the second term on the right-hand side of \eqref{eq:p-determinant} leads to the weak determinant problem.

\begin{ques}\label{ques:p-determinant-weak}
Let $0<p<2$ and $N'\leq N$. Is it true that 
\begin{equation}\label{eq:p-determinant-}
  R_{\mathcal U,v}^{(p)}[1,N]\geq R_{\mathcal U,v}^{(p)}[1,N']?
\end{equation}
\end{ques}
At $p=1$ and $p = 2$, the weak question is settled by the work \cite{B-1929, FMMZ-24, FHMNWZ-26-I} using the determinant representation of the volume: it holds in dimension at most three for $p =1$, and it holds in every dimension for $p=2$. When $p >2$, the counterexample was provided in \cite{FMMZ-24}.
The following theorem settles all remaining cases.

\begin{restatable}{theorem}{DetIneqWeak}\label{thm:determinant-ineq-weak}
The answer to Question~\ref{ques:p-determinant-weak} is affirmative for
$n=3$ and $0<p<1$.  It is negative for $n\geq3$ and $1<p<2$, and also for
$n\geq4$ and $0<p<1$.
\end{restatable}

The paper is organized as follows.  Section~\ref{sec:strong-lp-0} is devoted to the study of $L_p$ inequalities, and proves Theorem~\ref{thm:ratio-Lp-zonoid} and~\ref{thm:ratio-Lp-zonoidwekk}.
Next, in Section~\ref{sec:determinants}, we present the study of determinant-power inequalities, and prove Theorem~\ref {thm:determinant-ineq-2} and~\ref{thm:determinant-ineq-weak}. The appendix contains the auxiliary analytic estimates.

\section{Inequalities for the volume of projections of \texorpdfstring{$L_p$}{Lp} zonotopes} \label{sec:strong-lp-0}
\subsection{Failure of the strong \texorpdfstring{$L_p$}{Lp} projection inequality}\label{sec:strong-lp}
We begin with the following auxiliary lemma.
\begin{lemma} \label{lem:p-taylor-remainder}
    Let $r,s \in \R$. Then,
    \begin{equation}
    \label{eq:p-taylor-remainder}
    \left|
        |r+s|^p-|r|^p
        -p|r|^{p-1}\operatorname{sgn}(r)s
    \right|
    \leq 2^{2-p}|s|^p,
    \qquad 1<p<2.
\end{equation}
\end{lemma}

\begin{proof}
    Set
\[
    \phi(r)=|r|^{p-1}\operatorname{sgn}(r),
\]
with $\operatorname{sgn}(0) = 0$.
We first show that
\[
    |\phi(a)-\phi(b)|
    \leq
    2^{2-p}|a-b|^{p-1},
    \qquad a,b\in\mathbb R.
\]
If \(a\) and \(b\) have the same
sign, then
\[
    |\phi(a)-\phi(b)|=
    \left||a|^{p-1}-|b|^{p-1}\right|
    \leq
    |a-b|^{p-1},
\]
where we used $p-1 \in (0,1).$ If \(a\) and \(b\) have opposite signs, then the concavity of
\(x\mapsto x^{p-1}\) gives
\[
    |\phi(a)-\phi(b)|
    =
    |a|^{p-1}+|b|^{p-1}
    \leq
    2^{2-p}(|a|+|b|)^{p-1}
    =
    2^{2-p}|a-b|^{p-1}.
\]
The fundamental theorem of calculus gives
\[
    |r+s|^p-|r|^p
    -ps|r|^{p-1}\operatorname{sgn}(r)
    =
    ps\int_0^1
    \bigl(\phi(r+\lambda s)-\phi(r)\bigr)\,d\lambda.
\]
Therefore,
\[
\begin{aligned}
    \left|
        |r+s|^p-|r|^p
        -p|r|^{p-1}\operatorname{sgn}(r)s
    \right|\leq
    p|s|\int_0^1
    2^{2-p}|\lambda s|^{p-1}\,d\lambda
   =
    2^{2-p}|s|^p.
\end{aligned}
\]
\end{proof}
We shall also use the following integral representation for the mixed volume $V(M,N)$ of planar convex bodies $M$ and $N$; see \cite[Theorem~5.1.7 and (5.19)]{Schneider-14}:
$$
V(M,N)
=
\frac{1}{2}\int_{\mathbb S^1}
h_N(u)\,dS_M(u).
$$
In particular, $V(M,M)=|M|$. Moreover, Minkowski's first inequality \cite[Theorem~7.2.1]{Schneider-14} gives
$$
V(M,N)^2\geq |M|\,|N|.
$$

    \begingroup
\renewcommand{\MainTheoremExtra}{%
, that is,
        \[
            \left(\frac{\left|A \oplus_p B\right|}{\left|P_{e_2^{\perp}}\left(A \oplus_p B\right)\right|}\right)^p 
            <
            \left(\frac{|A|}{\left|P_{e_2^{\perp}} A\right|}\right)^p+\left(\frac{|B|}{\left|P_{e_2^{\perp}} B\right|}\right)^p.
        \]
}
\RatioLpZonoid*
\endgroup
    
    \begin{proof}
We start with the case $n=2$. Let 
    \[
     A = [-e_1,e_1] \oplus_p [-e_2,e_2], \quad B_t = [-e_1,e_1] \oplus_p [-e_1-te_2, e_1 + te_2], \quad   t  \geq 0.
    \]
    Note that 
    \begin{equation}
        \label{eq:PB}
        P_{e_2^\perp} B_t  = [-e_1,e_1] \oplus_p [-e_1,e_1] = 2^{\frac{1}{p}} P_{e_2^\perp} A. 
    \end{equation}
    Since $B_t = \begin{pmatrix}
            1&1\\
            0&t
        \end{pmatrix} A$,  $|B_t| = t |A|$. Therefore 
        \[
            \left(\frac{|A|}{\left|P_{e_2^{\perp}} A\right|}\right)^p+\left(\frac{|B_t|}{\left|P_{e_2^{\perp}} B_t\right|}\right)^p = \left(\frac{|A|}{\left|P_{e_2^{\perp}} A\right|}\right)^p \left( 1+ \frac{t^p}{2} \right).
        \]
    We recall that for any linear transformation $T$, $T(K \oplus_p L) =(TK) \oplus_p (TL)$. Next, using \eqref{eq:PB}, we obtain
    \[
        P_{e_2^{\perp}} (A \oplus_p B_t) = P_{e_2^{\perp}} A \oplus_p P_{e_2^{\perp}} B_t = P_{e_2^{\perp}} A \oplus_p (2^{\frac{1}{p}} P_{e_2^{\perp}} A) = 3^{\frac{1}{p}} P_{e_2^{\perp}} A.
    \]
    We claim that 
    \begin{equation} \label{eq:claim-lp-2}
        |A \oplus_p B_t| = |A \oplus_p B_0| + o(t^p).
    \end{equation}
    Let us first explain why this claim gives the desired planar counterexample. Observe that
    \begin{align}
        A \oplus_p B_0 &= [-e_1,e_1] \oplus_p [-e_2,e_2]\oplus_p [-e_1,e_1] \oplus_p [-e_1,e_1] = 3^{\frac{1}{p}}[-e_1,e_1] \oplus_p [-e_2,e_2]  
        \\
        &= \begin{pmatrix}
        3^{\frac{1}{p}}&0\\
        0&1
    \end{pmatrix} A.
    \end{align}
    Since the determinant of the matrix is $3^{\frac{1}{p}}$, $|A \oplus_p B_0| = 3^{\frac{1}{p}}|A|$. Thus,
    \[
        \left(\frac{\left|A \oplus_p B_t\right|}{\left|P_{e_2^{\perp}}\left(A \oplus_p B_t\right)\right|}\right)^p = \frac{\bigg(3^{\frac{1}{p}}|A| + o(t^p)\bigg)^p}{3|P_{e_2^{\perp}} A|^p} = \left(\frac{|A|}{\left|P_{e_2^{\perp}} A\right|}\right)^p + o(t^p).
    \]
    For sufficiently small $t$, we have
    \begin{align}
        \left( \frac{\left|A \oplus_p B_t\right|} {\left|P_{e_2^{\perp}}\left(A \oplus_p B_t\right)\right|} \right)^p 
        &=\left(\frac{|A|}{\left|P_{e_2^{\perp}} A\right|}\right)^p + o(t^p) < \left(\frac{|A|}{\left|P_{e_2^{\perp}} A\right|}\right)^p \left( 1+ \frac{t^p}{2} \right) 
        \\
        &=
        \left(\frac{|A|}{\left|P_{e_2^{\perp}} A\right|}\right)^p+\left(\frac{|B_t|}{\left|P_{e_2^{\perp}} B_t\right|}\right)^p.
    \end{align}
    This contradicts the inequality \eqref{eq:question1} in Question~\ref{ques:ratio-Lp-zonoid}. Hence, Question~\ref{ques:ratio-Lp-zonoid} in $\R^2$ has a negative answer. 

    We now prove the claim \eqref{eq:claim-lp-2}. For \(t>0\), define
\[
    T_t=
    \begin{pmatrix}
        1&0\\
        \frac{t}{3}&1
    \end{pmatrix},
    \qquad
    K_t=T_t^{-1}(A\oplus_pB_t).
\]
Thus, $ |K_t|=|A\oplus_pB_t|. $ Since 
\[
\begin{aligned}
    A\oplus_pB_t
    &=
    2^{1/p}[-e_1,e_1]
    \oplus_p[-e_2,e_2]\oplus_p
    [-(e_1+te_2),e_1+te_2],
\end{aligned}
\]
we obtain
\[
\begin{aligned}
    K_t&=  \begin{pmatrix}
        1&0\\
        -\frac{t}{3}&1
    \end{pmatrix}
    2^{1/p}[-e_1,e_1]
    \oplus_p[-e_2,e_2]\oplus_p
    [-(e_1+te_2),e_1+te_2]
    \\
    &=
    2^{1/p}
    \left[
        -\left(e_1-\frac{t}{3}e_2\right),
        e_1-\frac{t}{3}e_2
    \right]
    \oplus_p[-e_2,e_2]\oplus_p
    \left[
        -\left(e_1+\frac{2t}{3}e_2\right),
        e_1+\frac{2t}{3}e_2
    \right].
\end{aligned}
\]
Therefore, for \(u=(x,y)\in\mathbb S^1\),
\[
    h_{K_t}^p(u)
    =
    2\left|x-\frac{ty}{3}\right|^p
    +
    \left|x+\frac{2ty}{3}\right|^p
    +
    |y|^p.
\]
When \(t=0\), we have $K_0=A\oplus_pB_0$ and $h_{K_0}^p(u)=3|x|^p+|y|^p$.
Consequently,
\begin{equation}
    \label{eq:support-p-difference}
\begin{aligned}
    h_{K_t}^p(u)-h_{K_0}^p(u)
    &=
    2\left|x-\frac{ty}{3}\right|^p
    +
    \left|x+\frac{2ty}{3}\right|^p
    -
    3|x|^p.
\end{aligned}
\end{equation}
Since \(s\mapsto |s|^p\) is convex, we have
\(
    3|x|^p
    \leq
    2\left|x-\frac{ty}{3}\right|^p
    +
    \left|x+\frac{2ty}{3}\right|^p.
\)
It follows that
\begin{equation}
    \label{eq:support-p-positive}
    h_{K_t}^p(u)-h_{K_0}^p(u)\geq0
\end{equation}
and $ h_{K_t}(u)\geq h_{K_0}(u).$ 
Thus, $K_0 \subset K_t$ and 
\begin{equation}\label{lowebound}
    |A\oplus_pB_t|=|K_t|\ge |K_0| = |A\oplus_pB_0|.
    \end{equation}
We next prove an upper bound. 
Using Lemma \ref{lem:p-taylor-remainder}, we have
$$
        |r+s|^p \le |r|^p
        +p|r|^{p-1}\operatorname{sgn}(r)s+ 2^{2-p}|s|^p.
$$
We apply the above inequality to  \eqref{eq:support-p-difference}, first with
$
    (r,s)=(x,-\frac{ty}{3}),
$
and then with
    $(r,s)=(x,\frac{2ty}{3}).
$
It follows that
\[
 h_{K_t}^p(u)-h_{K_0}^p(u)\leq
    2^{2-p}\left(
        2\left|\frac{ty}{3}\right|^p
        +
        \left|\frac{2ty}{3}\right|^p
    \right)=
    \frac{2^{2-p}(2+2^p)}{3^p}|y|^pt^p\leq
    \frac{2^{2-p}(2+2^p)}{3^p}t^p.
\]
Thus,
\begin{equation}
    \label{eq:support-p-estimate}
    h_{K_t}^p(u)-h_{K_0}^p(u)
    \leq
    \frac{2^{2-p}(2+2^p)}{3^p}t^p
\end{equation}
uniformly on \(\mathbb S^1\).

We also need a pointwise estimate. Fix
\(u=(x,y)\in\mathbb S^1\) with \(x\neq0\), i.e. $u\not\in \{-e_2, e_2\}$. The function
\(s\mapsto |s|^p\) is twice continuously differentiable in a
neighborhood of \(x\), and
\[
    \frac{d^2}{ds^2}|s|^p
    =
    p(p-1)|s|^{p-2}.
\]
The second-order Taylor expansion at \(x\) gives
\[
\begin{aligned}
    2\left|x-\frac{ty}{3}\right|^p
    +
    \left|x+\frac{2ty}{3}\right|^p
    -
    3|x|^p
    &=
    \frac{p(p-1)}{3}|x|^{p-2}y^2t^2
    +o_u(t^2).
\end{aligned}
\]
In particular,
\[
    h_{K_t}^p(u)-h_{K_0}^p(u)=O_u(t^2).
\]
The condition \(x\neq0\) excludes only the two points
\((0,1)\) and \((0,-1)\) of \(\mathbb S^1\). Therefore,
\begin{equation}
    \label{eq:support-p-limit}
    \frac{h_{K_t}^p(u)-h_{K_0}^p(u)}{t^p}
    \longrightarrow0
    \qquad\text{ for }u\in\mathbb S^1 \setminus \{-e_2, e_2\}.
\end{equation}
Set
\[
    \delta_t=h_{K_t}-h_{K_0}.
\]
It follows from \eqref{eq:support-p-positive} that \(\delta_t\geq0\).  By concavity, the graph of the function $s \mapsto s^{1 / p}$ lies below its tangent line at $s=h_{K_0}^p(u)$. Therefore,
\[
\begin{aligned}
    \delta_t(u)
    &=
    \bigl(h_{K_t}^p(u)\bigr)^{1/p}
    -
    \bigl(h_{K_0}^p(u)\bigr)^{1/p}
    \leq
    \frac{1}{p}
    \bigl(h_{K_0}^p(u)\bigr)^{\frac1p-1}
    \bigl(h_{K_t}^p(u)-h_{K_0}^p(u)\bigr)
    \\
    &\leq
    \frac{1}{p}
    \bigl(h_{K_t}^p(u)-h_{K_0}^p(u)\bigr),
\end{aligned}
\]
where the second inequality follows from 
$
    h_{K_0}^p(u)\geq
    1 \mbox{ for all } u=(x,y)\in\mathbb S^1.
$
It follows from \eqref{eq:support-p-estimate} that
\begin{equation}
    \label{eq:delta-estimate}
    0
    \leq
    \delta_t(u)
    \leq
    \frac{2^{2-p}(2+2^p)}{p\,3^p}t^p
\end{equation}
uniformly on \(\mathbb S^1\). Moreover,
\eqref{eq:support-p-limit} implies
\begin{equation}
    \label{eq:delta-limit}
    \frac{\delta_t(u)}{t^p}
    \longrightarrow0
    \qquad\text{ for  }u\in\mathbb S^1 \setminus \{-e_2, e_2\}.
\end{equation}
Set
\[
    I_t
    =
    \int_{\mathbb S^1}\delta_t(u)\,dS_{K_0}(u).
\]
The body \(K_0\) is strictly convex; indeed, it is a linear image of
the unit ball of \(\ell_q^2\), where
$
    q>2.$  Therefore,
$ 
    S_{K_0}\bigl(\{e_2,-e_2\}\bigr)=0.
$
Consequently, \eqref{eq:delta-limit} holds
\(S_{K_0}\)-almost everywhere. Moreover, by
\eqref{eq:delta-estimate},
\[
    0
    \leq
    \frac{\delta_t(u)}{t^p}
    \leq
    \frac{2^{2-p}(2+2^p)}{p\,3^p}
\]
uniformly on \(\mathbb S^1\). Since \(S_{K_0}\) is a finite measure,
the dominated convergence theorem gives
\[
    I_t=o(t^p).
\]
Since \(h_{K_t}=h_{K_0}+\delta_t\), we obtain
\[
\begin{aligned}
    V(K_0,K_t)
    &=
    \frac12
    \int_{\mathbb S^1}
        h_{K_t}(u)\,dS_{K_0}(u)=
    \frac12
    \int_{\mathbb S^1}
        \bigl(h_{K_0}(u)+\delta_t(u)\bigr)
        \,dS_{K_0}(u)=
    |K_0|+\frac12 I_t.
\end{aligned}
\]
Minkowski's first inequality therefore gives
\[
    |K_0|\,|K_t|
    \leq
    V(K_0,K_t)^2
    =
    \left(
        |K_0|+\frac12 I_t
    \right)^2.
\]
Since \(|K_0|>0\), it follows that
\[
\begin{aligned}
    |K_t|
    &\leq
    \frac{
        \left(|K_0|+\frac12 I_t\right)^2
    }{|K_0|}
    =
    |K_0|+I_t+\frac{I_t^2}{4|K_0|}.
\end{aligned}
\]
Since  \(I_t=o(t^p)\) and \eqref{lowebound}, we  have $|K_t|-|K_0|=o(t^p)$.
This proves \eqref{eq:claim-lp-2}, and the proof of the theorem is complete in $\R^2$.

    We extend this counterexample to higher dimensions. Let $A,B$ be the $L_p$-zonotopes in $\R^2$ that give the negative answer to Question \ref{ques:ratio-Lp-zonoid}. Consider $A' = A \oplus_p [-e_3,e_3]\oplus_p \ldots \oplus_p [-e_n,e_n]$ and $B' =B \oplus_p [-e_3,e_3]\oplus_p \ldots \oplus_p[-e_n,e_n]$. Using \cite[Lemma 8]{FMMN-26-2} and \cite[Lemma  15]{MNZ-25}, we have $A' = A \oplus_p (\{0\}^2 \times B_q^{n-2})$ where $q$ is  the H\"older conjugate of $p$, that is $\frac{1}{p} + \frac{1}{q} = 1$. Thus,
    \[
        \frac{|A'|}{|P_{e_2^\perp}A'|} = \frac{\frac{\Gamma\left(1+\frac{2}{q}\right) \Gamma\left(1+\frac{n-2}{q}\right)}{\Gamma\left(1+\frac{n}{q}\right)}|A||B_q^{n-2}|}{\frac{\Gamma\left(1+\frac{1}{q}\right) \Gamma\left(1+\frac{n-2}{q}\right)}{\Gamma\left(1+\frac{n-1}{q}\right)}|P_{e_2^\perp}A||B_q^{n-2}|}
        =
        \frac{\Gamma\left(1+\frac{2}{q}\right) \Gamma\left(1+\frac{n-1}{q}\right)}{\Gamma\left(1+\frac{1}{q}\right) \Gamma\left(1+\frac{n}{q}\right)} \frac{|A|}{|P_{e_2^\perp}A|}= c_{n,q} \frac{|A|}{|P_{e_2^\perp}A|}.
    \]
    Therefore, the same computation gives 
    $$\frac{|B'|}{|P_{e_2^\perp}B'|} = c_{n,q} \frac{|B|}{|P_{e_2^\perp}B|} \qquad\mbox{ and }\qquad \frac{|A'\oplus_p B'|} {|P_{e_2^\perp}(A'\oplus_p B')|} = c_{n,q} \frac{|A\oplus_p B|} {|P_{e_2^\perp}(A\oplus_p B)|}. 
    $$
    The common factor $c_{n,q}$ cancels, and the result follows from the planar counterexample.
    \end{proof}

\subsection{The planar weak \texorpdfstring{$L_p$}{Lp} projection inequality}
\label{sec:planar-one-term-lp}

 We use H\"older estimates for maps closely related to the normalized duality map on $L_p$-spaces. These have been studied in several settings; see, for example, \cite[Chapter 9]{BL-00} for estimates between unit spheres of $L_p$-spaces, \cite{Ricard-15} for their noncommutative analogues. A result more directly comparable to the estimate needed here was proved by Carlen, Frank, and Lieb for the normalized duality map \cite[Lemma 3.3]{CFL-14}. We use the following sharp estimate in the real finite-dimensional setting, which generalizes the result of \cite{DW-64}.

\begin{lemma}
\label{lem:mazur-map-section}
Let $1<p\leq2$ and let $q$ be  the H\"older conjugate of $p$, that is $\frac{1}{p} + \frac{1}{q} = 1$. For a nonzero vector $u\in\R^m$, define
\[
    J(u)=\frac{(|u_1|^{p-2}u_1,\ldots,|u_m|^{p-2}u_m)}{\|u\|_p^{p-1}}
\]
to be the vector in $\partial B_q^m$ such that $\langle J(u),u\rangle=\|u\|_p$.
Then, for nonzero $u,v\in\R^m$,
\begin{equation}
\label{eq:mazur-map-estimate-section}
    \|J(u)-J(v)\|_q
    \leq
    2\left(
       \frac{\|u-v\|_p}{\|u\|_p+\|v\|_p}
    \right)^{p-1}.
\end{equation}
\end{lemma}
Here, we use the convention that $|0|^{p-2}0 = 0$.
The proof of Lemma~\ref{lem:mazur-map-section} is given in
Appendix~\ref{app:duality-map}.

\begin{lemma}
\label{lem:chord-speed-planar-section}
Let $1<p\leq2$ and let $q$ be  the H\"older conjugate of $p$. Let $A\subset\R^2$ be a planar $L_p$-zonoid with $P_{\R e_1}A=[-e_1,e_1]$. For $x\in[-1,1]$, let
\[
    L(x)=\big|\{y\in\R:(x,y)\in A\}\big|.
\]
Then, the function
\begin{equation}
\label{eq:chord-speed-monotone}
    x\mapsto \frac{L(x)}{(1-x^q)^{1/q}}
\end{equation}
is non-increasing on $[0,1)$.
\end{lemma}

\begin{proof}
Since $L$ is concave, it is differentiable almost everywhere. To prove the monotonicity, it is enough to prove at points where $L$ is differentiable that
\begin{equation}
\label{eq:chord-speed-differential}
    -L'(x)\geq x^{q-1}\bigl(L(x)-xL'(x)\bigr),\qquad 0<x<1.
\end{equation}
Indeed, this inequality implies that the derivative of the function \eqref{eq:chord-speed-monotone} is non-positive almost everywhere. Since that function is locally absolutely continuous on $(0,1)$, the desired monotonicity follows.
By approximation, it is enough to prove the assertion for
$L_p$-zonotopes. Thus, we write
\[
A=[-v_1,v_1] \oplus_p \cdots \oplus_p [-v_m,v_m], \qquad v_i=(a_i,b_i),
\]
and set $ a=(a_i)_{i=1}^m, b=(b_i)_{i=1}^m$. If $a,b$ are linearly dependent, then $v_1,\ldots,v_m$ all lie in a one-dimensional subspace. Since $P_{\R e_1}A=[-e_1,e_1]$, each vertical section of $A$ has length zero. Hence, $L(x) = 0$ for $x \in [0,1)$, and the desired monotonicity is trivial. We therefore assume that $a,b$ are linearly independent.
Then,
\[
    h_A(s,t)= \left(\sum_{i=1}^m |a_is+b_it|^p \right)^{1/p} = 
    \|sa+tb\|_p.
\]
Here, the norm is taken in $\R^m$.
Since $P_{\R e_1}A=[-e_1,e_1]$, the largest first coordinate of a point of $A$ is 1. Therefore
\begin{equation} \label{eq:norm-of-a}
    \|a\|_p = h_A(1,0) = \sup_{(u,v) \in A} u = 1.
\end{equation}
Define
$$
    H(s):=h_A(s,1)=\|sa+b\|_p.
$$
Recall the Legendre transform of $H$,
\[
    H^*(x):=\sup _{s \in \mathbb{R}}\{s x-H(s)\}.
\]
The function defined by
\[
    x \mapsto -\max\{y:(x,y)\in A\}, \qquad x \in [-1,1]
\]
and extended by $+\infty$ outside $[-1,1]$, is a proper closed convex function.
Thus,
\begin{align}
\sup _{x \in[-1,1]}(s x+\max \{y:(x, y) \in A\})  =\sup _{(x, y) \in A}(s x+y) 
& =h_A(s, 1) 
 =H(s) .
\end{align}
Hence, 
$
    \max \{y:(x, y) \in A\}=-H^*(x).
$
Since $A$ is origin-symmetric, 
\begin{equation} \label{eq:l-formula}
    L(x) = -H^*(x) - H^*(-x).
\end{equation}
Now fix $0 < x< 1$. Consider
\begin{equation} \label{eq:Legen-H}
     -H^*(x)= \inf _{s \in \mathbb{R}}\{H(s)-s x\}\quad
\text { and }\quad - H^*(-x) =\inf _{s \in \mathbb{R}}\{H(s)+s x\} .
\end{equation}
We show that in each infimum, there is a unique minimizer. Using the triangle inequality and \eqref{eq:norm-of-a}, we obtain 
$$
    H(s)=\|s a+b\|_p \geq |s|-\|b\|_p .
$$
Consequently,
$
H(s) \pm s x \geq(1-x)|s|-\|b\|_p .
$
Since $0<x<1$, we have
$
H(s)\pm s x \longrightarrow+\infty 
$
as $|s| \rightarrow \infty$. Hence, both functions attain their minimum. 
We next prove the uniqueness. For any distinct $s_1 \neq s_2$ and $0<\lambda<1$, we have
$$
H\left(\lambda s_1+(1-\lambda) s_2\right) =\left\|\lambda\left(s_1 a+b\right)+(1-\lambda)\left(s_2 a+b\right)\right\|_p \leq \lambda H\left(s_1\right)+(1-\lambda) H\left(s_2\right).
$$
If equality holds, there exists $c>0$ such that
$
s_1 a+b=c\left(s_2 a+b\right) .
$
Thus,
$$
\left(s_1-c s_2\right) a+(1-c) b=0 .
$$
Since $a$ and $b$ are linearly independent, this implies
$
c=1 $ and  $s_1=s_2,
$
which is a contradiction. Therefore, $H$ is strictly convex. Adding or subtracting a linear function preserves strict convexity. Hence, each of the two infima has exactly one minimizer. Denote them by
\[
    s_{+}=s_{+}(x)=\underset{s \in \mathbb{R}}{\operatorname{argmin}}\{H(s)-s x\}
    \quad \text { and }\quad
    s_{-}= s_{-}(x)=\underset{s \in \mathbb{R}}{\operatorname{argmin}}\{H(s)+s x\} .
\]
Equivalently,
\begin{equation} \label{eq:Legre-H-s}
-H^*(x)=H\left(s_{+}\right)-x s_{+} \quad \text { and }\quad
-H^*(-x)=H\left(s_{-}\right)+x s_{-} .
\end{equation}
Inserting them into \eqref{eq:l-formula}, we obtain
\begin{equation} \label{eq:l-formula-in-H}
    L(x)=H\left(s_{+}\right)+H\left(s_{-}\right)-x\left(s_{+}-s_{-}\right) .
\end{equation}
Now, we justify differentiability of $H$. For $p>1$, the $\ell_p$-norm is differentiable at every nonzero vector. Since $sa+b \neq 0$ for any $s$, $H$ is differentiable everywhere, with
\begin{equation} \label{eq:formula-H-prime}
H^{\prime}(s)=\frac{\sum_{i=1}^m a_i\left|s a_i+b_i\right|^{p-2}\left(s a_i+b_i\right)}{\|s a+b\|_p^{p-1}} .
\end{equation}
Here, $|0|^{p-2} 0$ is interpreted as $0$. Hence,
\begin{equation} \label{eq:H-prime-s}
    H^{\prime}\left(s_{+}\right)=x, \quad H^{\prime}\left(s_{-}\right)=-x .    
\end{equation}
We claim that
\begin{equation} \label{eq:l-prime-formula}
    L'(x)=-s_+ +s_-.
\end{equation}
Using \eqref{eq:l-formula}, it suffices to show that
$
    \left(H^*\right)^{\prime}(x)=s_{+}, $ and $\left(H^*\right)^{\prime}(-x)=s_{-} .
$
We first prove 
\begin{equation} \label{eq:general-formula}
    \partial H^*(x)=\left\{ \underset{s \in \mathbb{R}}{\operatorname{argmax}}\{x s-H(s)\} \right\}= \{s_+\}.
\end{equation}
For every $y \in \mathbb{R}$, the definition of Legendre transform and \eqref{eq:Legre-H-s} give
$$
H^*(y)  \geq y s_+-H(s_+) 
 =x s_+-H(s_+)+s_+(y-x) 
=H^*(x)+s_+(y-x).
$$
Therefore,  $s_+ \in \partial H^*(x)$. Conversely, suppose that $s \in \partial H^*(x)$. Then, for every $y \in \mathbb{R}$,
$$
H^*(y) \geq H^*(x)+s(y-x) .
$$
Rearranging, we get
$$
s y-H^*(y) \leq s x-H^*(x) .
$$
Taking the supremum over $y \in \mathbb{R}$ and using the fact that $H^{**} = H$, we get 
$
    H(s) \leq s x-H^*(x).
$
On the other hand, the definition of $H^*(x)$ always gives
$
H^*(x) \geq x s-H(s) .
$
Thus, equality holds,
\[
    H^*(x)=x s-H(s).
\]
Hence $s$ attains the maximum in the definition of $H^*(x)$. By uniqueness of the maximizer, we obtain $s=s_{+}$, verifying \eqref{eq:general-formula}.
Since $\partial H^*(x)$ is a singleton, $H^*$ is differentiable at $x$, and
$
\left(H^*\right)^{\prime}(x)=s_{+} .
$ 
The proof of $\left(H^*\right)^{\prime}(-x)=s_{-}$ is similar, so we omit it.
Therefore, we have proved \eqref{eq:l-prime-formula}. Combining this with \eqref{eq:l-formula-in-H}, we get
\[
    L(x)-x L^{\prime}(x)=H\left(s_{+}\right)+H\left(s_{-}\right).
\]
Thus \eqref{eq:chord-speed-differential} is equivalent to
\begin{equation}
\label{eq:delta-needed}
    s_+ - s_-\geq x^{q-1}\bigl(H(s_+)+H(s_-)\bigr).
\end{equation}
Since $H$ is convex, $H^{\prime}$ is non-decreasing. Since
$$
H^{\prime}\left(s_{-}\right)=-x<x=H^{\prime}\left(s_{+}\right),
$$
we must have
$
s_{-}<s_{+} .
$
Set
\[
    u=s_+a+b,
    \qquad
    v=s_-a+b.
\]
Then
\[
    H(s_+)=\|u\|_p,
    \qquad
    H(s_-)=\|v\|_p.
\]
Since $\|a\|_p=1$, we get
\[
    \|u-v\|_p=\|(s_+-s_-) a\|_p = s_+ -s_-.
\]
Let 
\[
    J(w)=\frac{(|w_1|^{p-2}w_1,\ldots,|w_m|^{p-2}w_m)}{\|w\|_p^{p-1}}.
\]
It follows from the formula of $H'$, \eqref{eq:formula-H-prime} that
\[
    H'(s)=\frac{\sum_{i=1}^m a_i\left|s a_i+b_i\right|^{p-2}\left(s a_i+b_i\right)}{\|s a+b\|_p^{p-1}} = \langle J(sa+b),a\rangle.
\]
Using \eqref{eq:H-prime-s}, H\"older's inequality and that $\|a\|_p=1$, we  obtain
\[
    2x=\langle J(u)-J(v),a\rangle
       \leq \|J(u)-J(v)\|_q\|a\|_p
       =\|J(u)-J(v)\|_q.
\]
By Lemma~\ref{lem:mazur-map-section}, we have
\[
    2x \leq \|J(u)-J(v)\|_q \leq 2\left(
       \frac{\|u-v\|_p}{\|u\|_p+\|v\|_p}
    \right)^{p-1} = 
    2\left(
       \frac{s_+-s_-}{\|u\|_p+\|v\|_p}
    \right)^{p-1}.
\]
Since $1/(p-1)=q-1$, we raise both sides to the power $q-1$, and obtain  \eqref{eq:delta-needed}. This implies \eqref{eq:chord-speed-differential}.
\end{proof}

\begin{prop}
\label{prop:horizontal-segment-planar-section}
Let $1<p\leq2$. Let $A\subset\R^2$ be a planar $L_p$-zonoid. For $\delta\geq0$, set
\[
    A_\delta=A\oplus_p[-\delta e_1,\delta e_1].
\]
Then
\[
    \frac{|A_\delta|}{|P_{\R e_1}A_\delta|}
    \geq
    \frac{|A|}{|P_{\R e_1}A|}.
\]
\end{prop}

\begin{proof}
By normalization, it is enough to prove when $P_{\R e_1}A=[-e_1,e_1]$.
Let $
    r_\delta=(1+\delta^p)^{1/p}.$ 
Observe that 
\begin{align}
    P_{\R e_1}A_\delta &=
    P_{\R e_1} ( A \oplus_p [-\delta e_1,\delta e_1])
    = 
    P_{\R e_1} A \oplus_p P_{\R e_1} [-\delta e_1,\delta e_1] 
    \\
    &= [-e_1,e_1] \oplus_p [-\delta e_1,\delta e_1] 
    = [-r_\delta e_1,r_\delta e_1].
\end{align}
The last equality follows directly from the definition of the $L_p$-sum. Indeed,
\begin{align}
h_{\left[-e_1, e_1\right] \oplus_p\left[-\delta e_1, \delta e_1\right]}^p (u) 
&=h_{\left[-e_1, e_1\right]}^p (u)+h_{\left[-\delta e_1, \delta e_1\right]}^p (u)=\left(1+\delta^p\right)\left|\left\langle u, e_1\right\rangle\right|^p 
\\
&= h_{[-(1+\delta^p)^{1/p}e_1,(1+\delta^p)^{1/p}e_1]}^p (u).
\end{align}
For $x \in\left[-r_\delta,r_\delta \right]$, define
$$
L_\delta(x)=\left|\left\{y \in \mathbb{R}:(x, y) \in A_\delta\right\}\right| .
$$
Similarly, for $x \in[-1,1]$, define
$$
L(x)=|\{y \in \mathbb{R}:(x, y) \in A\}| .
$$
We shall prove that
\begin{equation}
\label{eq:pointwise-chord-comparison}
    L_\delta(r_\delta x)\geq L(x),\qquad \text{ for }-1 < x<1.
\end{equation}
This gives  the desired inequality, since
\[
    |A_\delta|
    =\int_{-r_\delta}^{r_\delta}L_\delta(x)\,dx
    =r_\delta\int_{-1}^{1}L_\delta(r_\delta x)\,dx
    \geq r_\delta\int_{-1}^{1}L(x)\,dx
    =r_\delta |A|,
\]
while $|P_{\R e_1}A_\delta|=2r_\delta$ and $|P_{\R e_1}A|=2.$

Now, it remains to prove \eqref{eq:pointwise-chord-comparison}. By symmetry of $A$ and $A_\delta$, it is enough to consider the case $0\leq x < 1$. 
Using  another representation of $L_p$-sum introduced by Lutwak, Yang, and Zhang \cite{LYZ-12}, we have
\[
    A_\delta =  \bigcup_{t \in [0,1]} \left( (1-t)^{1/q} A + t^{1/q}[-\delta e_1,\delta e_1] \right) 
    = 
    \bigcup_{u \in [0,1]} \left( (1-u^q)^{1/q} A + u[-\delta e_1,\delta e_1] \right)  
    .
\]
We claim that
\[
    \bigcup_{u \in [0,1]} \left( (1-u^q)^{1/q} A + u[-\delta e_1,\delta e_1] \right)  
    = 
    \bigcup_{|s|\leq1}
    \left( (1-|s|^q)^{1/q}A+\delta s e_1 \right).
\]
First, let
$$
z \in \bigcup_{u \in[0,1]}\left(\left(1-u^q\right)^{1 / q} A+u\left[-\delta e_1, \delta e_1\right]\right) .
$$
Then, for some $u \in[0,1]$, there exists $s \in[-u, u]$ such that
$
z \in\left(1-u^q\right)^{1 / q} A+\delta s e_1 .
$
Since $|s| \leq u$, we have
$$
\left(1-u^q\right)^{1 / q} \leq\left(1-|s|^q\right)^{1 / q} \text {. }
$$
Because $0 \in A$ and $A$ is convex, this implies
$
\left(1-u^q\right)^{1 / q} A \subset\left(1-|s|^q\right)^{1 / q} A .
$
Therefore,
$$
z \in\left(1-|s|^q\right)^{1 / q} A+\delta s e_1,
$$
and hence $z$ belongs to the right-hand side. Conversely, let
$$
z \in \bigcup_{|s| \leq 1}\left(\left(1-|s|^q\right)^{1 / q} A+\delta s e_1\right).
$$
Choose $u=|s|$. Then, $u \in[0,1]$, and
$
\delta s e_1 \in u\left[-\delta e_1, \delta e_1\right] .
$
Thus,
$$
\left(1-|s|^q\right)^{1 / q} A+\delta s e_1 \subset\left(1-u^q\right)^{1 / q} A+u\left[-\delta e_1, \delta e_1\right] .
$$
Hence $z$ belongs to the left-hand side. This proves the claim. Hence,
\begin{equation}
\label{eq:Lp-sum-union-horizontal}
    A_\delta
    =
    \bigcup_{|s|\leq1}
    (1-|s|^q)^{1/q}A+\delta s e_1.
\end{equation}
Fix $0\leq x<1$. Consider the function
\[
    F(s)=\left|r_\delta x-\delta s\right|^q+|s|^q, \quad s \in \mathbb{R}.
\]
Since $q>1, $ the function $ F$ is strictly convex. Also,
$$
F(s) \geq|s|^q \longrightarrow \infty \quad \text { as }|s| \rightarrow \infty .
$$
Therefore, $F$ has a unique minimizer, say $s_0$. Because $q \geq 2$, the function $t \mapsto|t|^q$ is differentiable, with
$$
\frac{d}{d t}|t|^q=q|t|^{q-2} t.
$$
Hence
$$
F^{\prime}(s)=-q \delta|r_\delta x -\delta s|^{q-2}(r_\delta x-\delta s)+q|s|^{q-2} s .
$$
Set $y_0  = r_\delta x- \delta s_0$.  Then, the critical-point equation gives
\[
    \left|s_0\right|^{q-2} s_0=\delta\left|y_0\right|^{q-2} y_0.
\]
Since $\delta \geq 0$, this equation shows that $s_0$ and $y_0$ have the same sign. Taking absolute values and raising both sides to the power of $\frac{1}{q-1}$ gives
$
\left|s_0\right|=\delta^{1 /(q-1)}\left|y_0\right| .
$
Because $s_0$ and $y_0$ have the same sign,
$$
s_0=\delta^{1 /(q-1)} y_0 .
$$
Using that $y_0  = r_\delta x- \delta s_0$, and $r_\delta=\left(1+\delta^p\right)^{1 / p}$, we obtain
\[
    y_0=\frac{x}{\left(1+\delta^p\right)^{1 / q}}, \qquad 
s_0
=\frac{\delta^{p-1} x}{(1+\delta^p)^{1/q}}.
\]
Thus, $y_0$ and $s_0$ are non-negative with the properties that
$$
s_0^q=\frac{\delta^p x^q}{1+\delta^p} \leq \frac{\delta^p}{1+\delta^p}<1 , \quad \text{and} \quad y_0^q+s_0^q=\frac{x^q}{1+\delta^p}+\frac{\delta^p x^q}{1+\delta^p}=x^q.
$$
Let $$ y = \frac{y_0}{(1-s_0^q)^{1/q}}.$$ Then,
\[
    y^q
    =\frac{y_0^q}{1-s_0^q}
    =\frac{x^q-s_0^q}{1-s_0^q}
    \leq x^q,
\]
so $0 \leq y\leq x$. Also,
\begin{equation}
\label{eq:rho-y-identity}
    (1-s_0^q)(1-y^q)
    =1-s_0^q-y_0^q
    =1-x^q.
\end{equation}
Using \eqref{eq:Lp-sum-union-horizontal}, the set $(1-s_0^q)^{1/q} A+\delta s_0 e_1$ is contained in $A_\delta$. At the horizontal coordinate
\[
    r_\delta x=(1-|s_0|^q)^{1/q} y+\delta s_0,
\]
this set has vertical chord length $(1-s_0^q)^{1/q} L(y)$. Therefore,
\[
    L_\delta(r_\delta x)\geq (1-s_0^q)^{1/q} L(y).
\]
Lemma~\ref{lem:chord-speed-planar-section} gives
\[
    \frac{L(y)}{(1-y^q)^{1/q}}
    \geq
    \frac{L(x)}{(1-x^q)^{1/q}}.
\]
Using \eqref{eq:rho-y-identity}, we get
\[
    (1-s_0^q)^{1/q} L(y)
    \geq
    L(x)\frac{(1-s_0^q)^{1/q}(1-y^q)^{1/q}}{(1-x^q)^{1/q}}
    =L(x).
\]
Thus \eqref{eq:pointwise-chord-comparison} follows.
\end{proof}

\RatioLpZonoidweak*
\begin{proof}
The case $p=1$ follows from Bonnesen’s inequality. Thus, our goal is to show that for $p\in (1,2]$ and  $L_p$-zonoids $A, B$ in $\mathbb{R}^2$,  we have 
$$
\frac{\left|A \oplus_p B\right|}{\left|P_{u^{\perp}}\left(A \oplus_p B\right)\right|} 
\geq
\frac{|A|}{\left|P_{u^{\perp}} A\right|}.
$$
By rotation, it is enough to prove for $u=e_2$, so that $P_{u^\perp}=P_{\R e_1}$. By approximation, it is enough to prove when $B$ is an $L_p$-zonotope.
By an inductive argument,
it suffices to prove when $B$ is just one segment $[-v,v]$. Let $v=(\alpha,\beta)\in\R^2$. If $\alpha =0$, then $v$ is vertical, hence
\[
    P_{\R e_1}(A\oplus_p[-v,v])=P_{\R e_1}A,
\]
while $A\subset A\oplus_p[-v,v]$. Therefore,
\[
    \frac{\left|A \oplus_p B\right|}{\left|P_{u^{\perp}}\left(A \oplus_p B\right)\right|} 
\geq
\frac{|A|}{\left|P_{u^{\perp}} A\right|}.
\]
Assume that $\alpha\neq0$. Consider a transformation
\[
    S(x,y)=\left(x,y-\frac{\beta}{\alpha}x\right).
\]
Then,  $\det S=1$, the map $S$ preserves horizontal projection length, that is, $P_{\mathbb{R} e_1} S=P_{\mathbb{R} e_1}$.
Observe that
\[
    S(A\oplus_p[-v,v])=SA\oplus_p[-Sv,Sv] = SA \oplus_p [-\alpha e_1,\alpha e_1] 
    .
\]
Thus,
\[
    \frac{\left|A \oplus_p B\right|}{\left|P_{u^{\perp}}\left(A \oplus_p B\right)\right|}  = \frac{\left|SA \oplus_p [-\alpha e_1,\alpha e_1]\right|}{\left|P_{u^{\perp}}\left(SA \oplus_p [-\alpha e_1,\alpha e_1]\right)\right|}, \qquad \frac{|A|}{\left|P_{u^{\perp}} A\right|} =\frac{|SA|}{\left|P_{u^{\perp}} SA\right|}.
\]
It is enough to prove the monotonicity when the added segment is horizontal. After normalizing the horizontal projection to $[-e_1,e_1]$, this is exactly Proposition~\ref{prop:horizontal-segment-planar-section}. This proves the theorem for $1<p\leq2$.
\end{proof}

\section{Determinant inequalities: proofs and counterexamples}\label{sec:determinants}

We first prove the negative part of the strong two-term result and then its
planar positive part.  We subsequently treat the weak determinant  inequality, beginning
with the counterexamples and ending with the affirmative result in dimension
three.

\subsection{Counterexamples to the strong inequality}
\DetIneq*
\begin{proof}[Proof of the negative part of Theorem \ref{thm:determinant-ineq-2}]
We give two counterexamples in $\R^3$, according to the value of $p$, with
$N'=4$, $N=8$, and $v=e_3$.  It is enough to verify
\begin{equation}    \label{eq:ratio-matrix}
R_{\mathcal{U},e_3}^{(p)}[1,8]
<
R_{\mathcal{U},e_3}^{(p)}[1,4]
+
R_{\mathcal{U},e_3}^{(p)}[5,8]
.
\end{equation}
We start with the counterexample for $p \in (1,2)$. We consider $\mathcal{U}=(u_i)_{i=1}^8$ where the vectors are given by the columns of the following matrix, 
\[
    \begin{pmatrix}
        1&0&1&0     &-1&-1&-1&-1\\
        0&1&0&1     &-1&-1&0&0\\
        1&1&-1&-1   &-1&1&-1&1\\
    \end{pmatrix}.
\]
We compute 
\begin{center}
    $R_{\mathcal{U},e_3}^{(p)}[1,4] =R_{\mathcal{U},e_3}^{(p)}[5,8] =2^p $ and $R_{\mathcal{U},e_3}^{(p)}[1,8]  = \frac{3+7\cdot 2^p +3^p}{5}$. 
\end{center}
Inserting them into \eqref{eq:ratio-matrix}, we need to check
\[
    2^p +2^p > \frac{3+7\cdot 2^p +3^p}{5}, \quad p \in (1,2).
\]
Equivalently, for $p \in (1,2)$, $2^p > 1 + 3^{p-1}.$ We consider the function $\phi $ on $(0,\infty)$ defined by $ \phi(x) = x^{p-1}$. Since $0<p-1<1$, the function $\phi $ is strictly concave, and so 
\[
    2^p = 2\phi(2) > \phi(1) + \phi(3) = 1 + 3^{p-1}.
\]

Next, we show the counterexample for $ p \in (0,1]$.  Consider $\mathcal{U}=(u_i)_{i=1}^8$ given by the columns of the following matrix, 
\[
    \begin{pmatrix}
        1&2&2&1&2&1&0&0\\
        0&0&-2&-1&0&0&1&2\\
        3&-1&0&3&-1&3&3&-1\\
    \end{pmatrix}.
\]
The absolute determinant data are the same as in the strong-projection counterexample in \cite{FHMNWZ-26-I}.
We will check \eqref{eq:ratio-matrix}. After a direct computation and simplification, we obtain
\begin{equation}
    R_{\mathcal{U},e_3}^{(p)}[1,4] +
    R_{\mathcal{U},e_3}^{(p)}[5,8] -R_{\mathcal{U},e_3}^{(p)}[1,8]
    = \frac{2}{5} \cdot \frac{ 6^p+2\cdot7^p-1-3^p-13^p }{ 1+2\cdot2^p+4^p }.
\end{equation}
For $p\in (0,1]$ we note that 
$$
 6^p+2\cdot 7^p-1 -3^p-13^p >  2\cdot 7^p-1 -13^p \ge 0, 
$$
where the last inequality follows again from the fact that for $0\leq p \leq 1$, the function $x \mapsto x^{p}$ is concave. 

We now clarify how to generalize the counterexample to higher dimensions by considering 
$$
    (\quad u_1 \quad \ldots \quad u_4 \quad e_4 \quad \ldots \quad e_n \quad u_5 \quad \ldots \quad u_8 \quad e_4 \quad \ldots \quad e_n \quad ),
$$
where $u_1,\ldots,u_8,v$ are embedded to $\R^3 \times \R^{n-3}$, $N' = n+1$ and $N = 2n+2$. The computation is the same, so we omit the proof. \qedhere

\end{proof}
\subsection{The planar strong inequality}
\label{sec:affirmative-cases}
\begin{lemma} \label{lem:laplace}
    For $p \in (0,2)$, one has 
    \[
        |x|^p = c_p \int_{0}^\infty \frac{1 - \cos (tx)}{t^{p+1}} dt,
    \]
    where $c_p >0$ is a constant depending on $p$.
\end{lemma}

\begin{proof}
The case  $x=0$ is trivial. Assume that $x \neq 0.$ Performing the variable substitution $u = |x| t$, we have
    \begin{align}
        \int_{0}^\infty \frac{1 - \cos (tx)}{t^{p+1}} dt &= |x|^p\int_{0}^\infty \frac{1 - \cos u}{u^{p+1}} du.
    \end{align}
    Note that $c_p^{-1}:=\int_{0}^\infty \frac{1 - \cos u}{u^{p+1}} du$ is finite and positive when $p \in (0,2)$.
\end{proof}
\begin{proof}[Proof of the positive part of Theorem \ref{thm:determinant-ineq-2}]
 By applying an orthogonal transformation, we may assume that
$v=e_1$.   For each $1 \leq i \leq N$, we write $u_i = (x_i,y_i)$. We prove the following inequality: for any $ p \in (0,2) $
    \begin{equation} \label{eq:ineq-determinant}
        \frac{\displaystyle\sum_{1\leq i< j \leq N'} \left| x_iy_j-x_jy_i\right|^p}{\displaystyle\sum_{1\leq i \leq N'}|y_i|^p} 
        +
        \frac{\displaystyle\sum_{N' + 1\leq i< j \leq N} \left| x_iy_j-x_jy_i\right|^p}{\displaystyle\sum_{N'+1\leq i \leq N}|y_i|^p} 
        \leq 
        \frac{\displaystyle\sum_{1\leq i< j \leq N} \left| x_iy_j-x_jy_i\right|^p}{\displaystyle\sum_{1\leq i \leq N}|y_i|^p}.
    \end{equation}
    In what follows, all indices are assumed to range from $1$ to $N$. First, suppose that $\sum_{i \leq N^{\prime}}\left|y_i\right|^p=0 $ or $ \sum_{i>N^{\prime}}\left|y_i\right|^p=0$. Then,  \eqref{eq:ineq-determinant} is immediate. Indeed, without loss of generality, $\sum_{i \leq N^{\prime}}\left|y_i\right|^p=0.$ Thus, all $y_i=0$ for $i \leq N'$, the first ratio is zero by convention. Moreover, the denominator on the right-hand side is the same as the denominator on the left-hand side, while the numerator on the right-hand side contains additional nonnegative terms. Therefore, we may assume that both partial denominators are positive.
    Observe that 
    \[
        \frac{\displaystyle\sum_{i< j \leq N'} \left| x_iy_j-x_jy_i\right|^p}{\displaystyle\sum_{ i \leq N'}|y_i|^p} 
        =
        \sum_{i \leq N': y_i =0}|x_i|^p+\frac{  \displaystyle\sum_{i< j \leq N': y_i, y_j \neq 0} \left| x_iy_j-x_jy_i\right|^p}{ \displaystyle\sum_{ i \leq N': y_i \neq 0} |y_i|^p}.
    \]
    Computing the remaining quotients and substituting them into \eqref{eq:ineq-determinant}, the term $ \sum_{i: y_i =0}|x_i|^p$ can be canceled from both sides. Hence, it is enough to prove the inequality \eqref{eq:ineq-determinant} under the assumption that $y_i \neq 0$ for all $i$.
    Simplifying inequality \eqref{eq:ineq-determinant}, it is equivalent to
    \begin{equation}
\label{eq:equivalent-inequality}
\frac{\displaystyle\sum_{1\le i<j\le N'}
      \!\!\!|x_i y_j-x_j y_i|^p}
     {\left(\displaystyle\sum_{1\le i\le N'}|y_i|^p\right)^2}
+
\frac{\displaystyle\sum_{N'+1\le i<j\le N}
    \!\!\!  |x_i y_j-x_j y_i|^p}
     {\left(\displaystyle\sum_{N'+1\le i\le N}|y_i|^p\right)^2}
\leq
\frac{\displaystyle\sum_{\substack{1\le i\le N'\\
                N'+1\le j\le N}}
     \!\!\!\!\!\! |x_i y_j-x_j y_i|^p}
     {\left(\displaystyle\sum_{1\le i\le N'}\!\!\!|y_i|^p\right)\!\!
      \left(\displaystyle\sum_{N'+1\le i\le N}\!\!\!|y_i|^p\right)}.
\end{equation}
    We define two probability measures $\mu$ and $\nu $ on $\R$ as follows:
    \[
        \mu = \frac{\displaystyle\sum_{i \leq N'}|y_i|^p \delta_{\{x_i/y_i\}}}{\displaystyle\sum_{i \leq N'}|y_i|^p}  \qquad \text { and } \qquad 
        \nu = \frac{\displaystyle\sum_{N'+1\leq i }|y_i|^p \delta_{\{x_i/y_i\}}}{\displaystyle\sum_{N'+1\leq i }|y_i|^p} .
    \]
    Using the measures $\mu $ and $\nu$, we can rewrite \eqref{eq:equivalent-inequality} as
    \begin{equation}
        \label{eq:equivalent-with-measures}
    \frac{1}{2} \int \int |s-t|^p d \mu(s) d \mu(t) +\frac{1}{2} \int \int |s-t|^p d \nu(s) d \nu(t) \leq \int \int |s-t|^p d \mu(s) d \nu(t).
    \end{equation}
    Denote the signed measure $\sigma = \mu -\nu$. Thus, \eqref{eq:equivalent-with-measures} is equivalent to
    \begin{equation} \label{eq:equivalent-with-sign-measure}
        \int \int |s-t|^p d \sigma(s) d \sigma(t) \leq 0.
    \end{equation}
    Using Lemma \ref{lem:laplace} and Fubini's theorem, we have
    \begin{align}
        \int \int |s-t|^p d \sigma(s) d \sigma(t) &= \int \int \left(
        c_p \int_{0}^\infty \frac{1 - \cos ((s-t)u)}{u^{p+1}} du \right)
        d \sigma(s) d \sigma(t)
        \\
        &=  \int_{0}^\infty \frac{c_p}{u^{1+p}}\int \int 
        (1 - \cos ((s-t)u)) 
        d \sigma(s) d \sigma(t) du.
    \end{align}
    Since $\mu,\nu$ are probability measures, $\sigma(\R) = 0$. Thus,
    \begin{align}
        \int &\int 
        (1 - \cos ((s-t)u)) 
        d \sigma(s) d \sigma(t) 
        = -\int \int 
        \cos ((s-t)u) 
        d \sigma(s) d \sigma(t)
        \\
        &= -\int \int 
        \Re \{e^{iu(s-t)}\}
        d \sigma(s) d \sigma(t)
        = -\Re \left\{\int \int 
        e^{iu(s-t)}
        d \sigma(s) d \sigma(t)\right\}
        = - \left|\int 
        e^{ius} d \sigma(s) \right|^2 \\
        &\leq 0,
    \end{align}
    where $\Re f$ is the real part of $f$. Plugging it back into \eqref{eq:equivalent-with-sign-measure}, the proof is complete.
\end{proof}

\subsection{Counterexamples to the weak inequality}
\DetIneqWeak*
\begin{proof}[Proof of the negative part of Theorem \ref{thm:determinant-ineq-weak}]
    We give two counterexamples, according to the value of $p$.
    We first consider the case $p \in (1,2)$. Let $N'=5$ and $\mathcal{U}=(u_i)_{i=1}^6$ given by the columns of the following matrices,
\[
    \begin{pmatrix}
        1&1&1&1&\varepsilon^{1/p}     &0    \\
        -1&1&-1&1&0             &1    \\
        0&0&1&1&2 \varepsilon^{1/p}   &0    \\
    \end{pmatrix},
\]
and take $v = e_3$. We will prove that there exists $\varepsilon = \varepsilon (p)> 0$ such that
\begin{equation}    \label{eq:ratio-matrix-weak}
R_{\mathcal{U},e_3}^{(p)}[1,5]
>
R_{\mathcal{U},e_3}^{(p)}[1,6].
\end{equation}
We compute 
    $$
        R_{\mathcal{U},e_3}^{(p)}[1,5] = \displaystyle\frac{2^{p+2}+\left(4^p+2 \cdot 3^p+2^p+2\right) \varepsilon}{2^{p+2}+4 \varepsilon},
    $$
    and
    $$
        R_{\mathcal{U},e_3}^{(p)}[1,6] = \displaystyle\frac{2^{p+2}+4+\left(4^p+2 \cdot 3^p+3\cdot 2^p+4\right) \varepsilon}{2^{p+2}+4+5 \varepsilon}.
    $$
Thus, \eqref{eq:ratio-matrix-weak} is equivalent to
\[
    \varepsilon\left(4\left(2 \cdot 3^p-4^p-2\right)+\left(4^p+2 \cdot 3^p-7 \cdot 2^p-6\right) \varepsilon\right) > 0.
\]
Define the function $F$ by
\[
    F(\varepsilon) = \varepsilon\left(4\left(2 \cdot 3^p-4^p-2\right)+\left(4^p+2 \cdot 3^p-7 \cdot 2^p-6\right) \varepsilon\right).
\]
Since $p \in (1,2)$, using the concavity of $x \mapsto x^{p-1}$, we obtain $3^{p-1} > \frac{2}{3}4^{p-1}+\frac{1}{3}$. Hence,
\[
    F' (0) = 4\left(2 \cdot 3^p-4^p-2\right) >0.
\]
Since $F(0) = 0$, there exists $\varepsilon>0$ such that $ F(\varepsilon) >0$, verifying \eqref{eq:ratio-matrix-weak}. This counterexample can be generalized to higher dimensions by considering 
$$
    (\quad u_1 \quad \ldots \quad u_5 \quad e_4 \quad \ldots \quad e_n \quad u_6 \quad ),
$$
where $u_1,\ldots,u_6, v$ are embedded to $\R^3 \times \R^{n-3}$, $N' = n+2$ and $N = n+3$.

Next, we present the counterexample in dimension $n \geq 4$ for $p \in (0,1)$. Consider $N'=6$ and $\mathcal{U}=(u_i)_{i=1}^7$ given by the columns of the following matrices,
\begin{equation}
    \begin{pmatrix}
        1&0&1&1&0&0&1\\
        0&1&1&0&0&0&0\\
        0&1&1&-1&-3&0&0\\
        2&2&2&0&0&-1&0
    \end{pmatrix},
\end{equation}
and take $ v = e_2$. We compute 
\[
    R_{\mathcal{U},e_2}^{(p)}[1,6] = \frac{3+2^p+5 \cdot 3^p+3 \cdot 6^p}{5+4 \cdot 2^p+3 \cdot 3^p+5 \cdot 6^p} \quad \text{and} \quad R_{\mathcal{U},e_2}^{(p)}[1,7] = \frac{5+3 \cdot 2^p+7 \cdot 3^p+5 \cdot 6^p}{8+9 \cdot 2^p+4 \cdot 3^p+8 \cdot 6^p}.
\]
We need to show that
\[
    R_{\mathcal{U},e_2}^{(p)}[1,6]
>
R_{\mathcal{U},e_2}^{(p)}[1,7],
\]
or equivalently,
\begin{equation} \label{eq:deter-counterexample-01-weak}
    -1 +2\cdot 3^p -3 \cdot 4^p+ 10 \cdot 6^p - 9^p + 2 \cdot 18^p - 36^p >0.
\end{equation}
Then, $-1 +2\cdot 3^p >1 >0$.
Since $p \in (0,1)$, we get $2 \cdot 18^p - 36^p = 18^p (2 - 2^p) >0$. Since $9^p = 6^p \left(\frac{3}{2}\right)^p <\frac{3}{2} \cdot 6^p$, we obtain
\[
    -3 \cdot 4^p+ 10 \cdot 6^p - 9^p > 10 \cdot 6^p -\frac{3}{2} \cdot 6^p - 3 \cdot 6^p = \frac{11}{2} \cdot 6^p.
\]
Combining these three inequalities yields \eqref{eq:deter-counterexample-01-weak}. This counterexample can be generalized to higher dimensions by considering 
$$
    (\quad u_1 \quad \ldots \quad u_6 \quad e_5 \quad \ldots \quad e_n \quad u_7 \quad ),
$$
where $u_1,\ldots,u_7,v$ are embedded to $\R^4 \times \R^{n-4}$, $N' = n+2$ and $N = n+3$.
\end{proof}
\subsection{The weak determinant inequality in dimension three}

\begin{lemma} \label{lem:representation-p-01}
    Let $ a,b \in \R$ and $ p \in (0,1)$. Then,
    \[
        |a-b|^p=\frac{p(1-p)}{2} \int_0^{\infty} \int_{\mathbb{R}}\left|{1}_{[s, s+t]}(a)-{1}_{[s, s+t]}(b)\right| d s \, t^{p-2} d t.
    \]
\end{lemma}

\begin{proof}
    The formula is trivial when $a= b$. Hence, we may assume that $a <b$. Observe that
    \[
        1_{[s,s+t]} (a) = 1 \qquad \text{if and only if} \quad s \in [a-t,a].
    \]
    Thus,
    \[
        \left|{1}_{[s, s+t]}(a)-{1}_{[s, s+t]}(b)\right|= 1_{[a-t, a] \triangle[b-t, b]}(s).
    \]
    Hence, 
    \[
        \int_{\mathbb{R}}\left|{1}_{[s, s+t]}(a)-{1}_{[s, s+t]}(b)\right| d s =
        2 \min \{t,b-a\}.
    \]
    Therefore,
    \begin{align}
        &\frac{p(1-p)}{2} \int_0^{\infty} \int_{\mathbb{R}}\left|{1}_{[s, s+t]}(a)-{1}_{[s, s+t]}(b)\right| d s \, t^{p-2} d t 
        \\
        &=
        p(1-p) \left(\int_0^{b-a}  t^{p-1} d t
        +
        (b-a)\int_{b-a}^{\infty}   t^{p-2} d t\right)
        =
        p(1-p) \left( \frac{(b-a)^p}{p}
        +
        \frac{(b-a)^{p}}{1-p}\right)
        \\
        &= (b-a)^p. \qedhere
    \end{align}
\end{proof}

    \begin{lemma} \label{lem:desired-ineq-12-01}
    For any $ 1 \leq i \leq 3 $, let $ x_i, y_i \in \{0,1\}$. 
    For each $ 1 \leq i \leq 3$, let $j,k$ be the remaining two indices, and define
    \[
        X_ i = \frac{|x_i - x_j| + |x_i - x_k| - |x_k - x_j|}{2}, \quad \text{and} \quad Y_ i = \frac{|y_i - y_j| + |y_i - y_k| - |y_k - y_j|}{2}.
    \]
    Then, 
    \begin{equation}
    \begin{split} 
        &\left|\operatorname{det}\left(\begin{array}{ccc}1 & 1 & 1 \\ x_1 & x_2 & x_3 \\ y_1 & y_2 & y_3\end{array}\right)\right| 
        =
        \sum_{i \neq j} X_i Y_j.
    \end{split} 
    \end{equation}
\end{lemma}

\begin{proof}
    If $x_1 = x_2 = x_3$ or $y_1 = y_2 = y_3$, then the row corresponding to $x$ or $y$ is a multiple of the first row. Hence, the determinant is zero. Also, in this case, it is straightforward to check that $X_1 = X_2 = X_3 = 0$ or $Y_1 = Y_2 = Y_3 = 0$, respectively. Thus, the right-hand side is also zero.

    We therefore assume that $x_1,x_2,x_3$ are not all equal, and that $y_1,y_2,y_3$ are not all equal. Since each entry is either $0$ or $1$, there exists an index $m\in \{1,2,3\}$ such that $x_m$ is different from the other two $x$'s. Thus,
    \begin{equation} \label{eq:A-formula}
        X_i =
        \begin{cases}
            \frac{1+1-0}{2} = 1, & i=m,
            \\
            \frac{1+0-1}{2} =0, &i \neq m.
        \end{cases} 
    \end{equation}
    Similarly, there exists an index $n$ such that $y_n$ is different from the other two $y$'s and 
    \begin{equation} \label{eq:B-formula}
        Y_i =
        \begin{cases}
            1, & i=n,
            \\
            0, &i \neq n.
        \end{cases} 
    \end{equation}
    \begin{enumerate}[wide =0pt, labelwidth= 1.1cm]
        \item[\bf Case 1:] Suppose that $m=n$. The columns corresponding to other indices are the same. Thus, the determinant is zero. On the other hand, using \eqref{eq:A-formula} and \eqref{eq:B-formula}, we have 
        \[
            \sum_{i \neq j} X_i Y_j = \sum_{i \in \{1,2,3\} \setminus \{m\}} Y_i =0. 
        \]
        \item[\bf Case 2:] Suppose $ m \neq n$. First, using \eqref{eq:A-formula} and \eqref{eq:B-formula}, we have 
        $\sum_{i \neq j} X_i Y_j =1$.
        On the other hand, if $x_m =0$, we replace the second row by the first row minus the second row. Since this process changes the determinant only by a sign, we obtain
        \[
            \left|\operatorname{det}\left(\begin{array}{ccc}1 & 1 & 1 \\ x_1 & x_2 & x_3 \\ y_1 & y_2 & y_3\end{array}\right)\right|  = 
            \left|\operatorname{det}\left(\begin{array}{ccc}1 & 1 & 1 \\ 1- x_1 & 1- x_2  & 1- x_3 \\ y_1 & y_2 & y_3\end{array}\right)\right|.
        \]
        Therefore, after this replacement, we may assume that the second row has value 1 in the $m$th entry and value 0 in the remaining entries. Applying the same argument to the third row, we may assume that the third row has value 1 in the $n$th entry and value 0 in the remaining entries. Since $m \neq n$, the resulting determinant has absolute value  1. \qedhere
    \end{enumerate}
\end{proof}

\begin{lemma} \label{lem:desired-ineq-12}
    Let $ p \in (0,1) $. For any $ 1 \leq i \leq 3 $, let $ a_i, b_i \in \R$. Define the measure $\mu$ on $\R \times (0,\infty)$ by 
    $$
        d \mu (s,t) = \frac{p(1-p)}{2} d s \, t^{p-2} d t.
    $$ 
    Then, 
    \begin{equation}
    \begin{split} 
        &\left|\operatorname{det}\left(\begin{array}{ccc}1 & 1 & 1 \\ a_1 & a_2 & a_3 \\ b_1 & b_2 & b_3\end{array}\right)\right|^p  
        \\
        &\qquad \leq
        \iint\left|\operatorname{det}\left(\begin{array}{ccc}1 & 1 & 1 \\ 1_{[s,s+t]}(a_1) & 1_{[s,s+t]}(a_2) & 1_{[s,s+t]}(a_3) \\ 1_{[\sigma,\sigma+\tau]}(b_1) & 1_{[\sigma,\sigma+\tau]}(b_2) & 1_{[\sigma,\sigma+\tau]}(b_3)
        \end{array}\right)\right| d \mu(s, t) d \mu(\sigma, \tau).
    \end{split} \label{eq:desired-ineq-12}
    \end{equation}
\end{lemma}

\begin{proof}
    For each $i \in \{1,2,3\}$, let $j,k$ denote the remaining two indices, and define
    \[
        A_ i = \frac{|a_i - a_j|^p + |a_i - a_k|^p - |a_k - a_j|^p}{2}, \quad \text{and}\quad B_ i = \frac{|b_i - b_j|^p + |b_i - b_k|^p - |b_k - b_j|^p}{2}.
    \]
    First, we claim that
    \begin{equation} \label{eq:desired-ineq-12-claim1}
        \left|\operatorname{det}\left(\begin{array}{ccc}1 & 1 & 1 \\ a_1 & a_2 & a_3 \\ b_1 & b_2 & b_3\end{array}\right)\right|^p  
        \leq
        \sum_{i \neq j} A_i B_j.
    \end{equation}
    Since $|a-b|^p$ is a metric for $p \in (0,1]$, we get that 
        $A_i,B_i \geq 0$ for all $1 \leq i\leq 3$.

    We may assume that $a_1,a_2,a_3$ are not all equal and that $b_1,b_2,b_3$ are not all equal. Indeed, if $a_1 =a_2=a_3$, then the second row of the determinant is a multiple of the first row, so the determinant is zero and inequality \eqref{eq:desired-ineq-12-claim1} is trivial. The same argument applies if $b_1=b_2=b_3$. 
    Without loss of generality, reordering if needed, we assume that $a_1 \leq a_2 \leq a_3$. Thus, $a_3 > a_1$. Using the first row of the matrix to eliminate the first entry of the second row, we get
        \begin{align}
            \left|\operatorname{det}\!\!\left(\begin{array}{ccc}1 & 1 & 1 \\ a_1 & a_2 & a_3 \\ b_1 & b_2 & b_3\end{array}\right)\right|^p \!\!  
            = \left|\operatorname{det}\!\!\left(\begin{array}{ccc}1 & 1 & 1 \\ 0 & a_2-a_1 & a_3-a_1 \\ b_1 & b_2 & b_3\end{array}\right)\right|^p 
            \!\!\!= |a_3 -a_1|^p \left|\operatorname{det}\!\!\left(\begin{array}{ccc}1 & 1 & 1 \\ \tilde a_1 &\tilde a_2 & \tilde a_3 \\ b_1 & b_2 & b_3\end{array}\right)\right|^p\!\!\!,
        \end{align}
        where 
        \[
            \tilde a_1=0,\qquad \tilde a_2=\frac{a_2-a_1}{a_3-a_1},\qquad \tilde a_3=1.
        \]
        Then, for every $i, j \in\{1,2,3\}$,
        $$
        \left|a_i-a_j\right|^p=\left|a_3-a_1\right|^p\left|\widetilde{a}_i-\widetilde{a}_j\right|^p .
        $$
        On the other hand, for each  $1\leq i \leq 3$, we have 
        \[
            A_i = |a_3 -a_1|^p \frac{| \tilde a_i - \tilde a_j|^p + |\tilde a_i - \tilde a_k|^p - | \tilde a_k - \tilde a_j|^p}{2}.
        \]
        Dividing both sides of \eqref{eq:desired-ineq-12-claim1} by $|a_3 -a_1|^p$, it suffices to prove the inequality in the case
        \[
            a_1 = 0, \qquad a_2 = r, \qquad a_3 =1\qquad \text{for } r \in [0,1].
        \]
        We normalized $b_1,b_2,b_3$ in the same way, and we may assume that
        \[
            \{ b_1,b_2,b_3\} = \{ 0,s,1\} \quad \text{for } s \in [0,1].
        \]
        Note that the indices of $b_1,b_2,b_3$ remain attached to the corresponding $a_1,a_2,a_3$. Thus, we cannot reorder them independently. 
    \begin{enumerate}[wide =0pt, labelwidth= 1.1cm]
        \item[\bf Case 1:] Suppose that $b_1 = 0, b_2 = s, b_3 =1 $ for $ s \in [0,1]$.
        Since both sides of the desired inequality are symmetric under interchanging the $a$- and $b$-coordinates, it is enough to consider the case
        $$
        0 \leq s \leq r \leq 1 .
        $$
        The left-hand side of \eqref{eq:desired-ineq-12-claim1} is
        \begin{align}
            \left|\operatorname{det}\left(\begin{array}{ccc}1 & 1 & 1 \\ 0 & r & 1 \\ 0 & s & 1\end{array}\right)\right|^p   
            &= (r-s)^p.
        \end{align}
        
        Moreover, 
            $$A_1 = \frac{r^p + 1 - (1-r)^p}{2}, \quad A_2 = \frac{r^p + (1-r)^p - 1}{2}, \quad A_3 = \frac{1 + (1-r)^p-r^p}{2}, $$
        and 
            $$B_1 = \frac{s^p + 1 - (1-s)^p}{2}, \quad B_2 = \frac{s^p + (1-s)^p - 1}{2},\quad B_3 = \frac{1 + (1-s)^p-s^p}{2} .$$
        Thus,
            $\sum_{i \neq j} A_i B_j = A_1 (1-s)^p + A_2 + A_3s^p ,$
        it remains to prove that $$A_1 (1-s)^p + A_2 + A_3s^p \geq |r-s|^p.$$ 
        For fixed $r$, define $F$ on $[0,r]$ by
        \[
            F(s) = A_1 (1-s)^p + A_2 + A_3s^p - (r-s)^p.
        \]
        Then, \eqref{eq:desired-ineq-12-claim1} is equivalent to $F(s) \geq 0$. It is enough to show that $F$ is increasing since $F(0) = A_1+A_2-r^p = 0$.
        We compute
        $$
            F' (s) =  -pA_1(1-s)^{p-1} +pA_3s^{p-1} + p(r-s)^{p-1}.
        $$
        Since $A_1 = 1-A_3$ and $A_1,A_3 \geq 0$, we get that $0\leq A_1 \leq 1$. Thus,
        \[
            pA_1(1-s)^{p-1} \leq p(1-s)^{p-1} \leq p(r-s)^{p-1},
        \]
        where the last inequality follows from the fact that $ p<1$ and $r-s \leq 1-s$. Hence $F' \geq 0$. 
        \item[\bf Case 2:] Suppose that $b_1 = 0, b_2 =1,b_3 =s$. 
        The left-hand side of \eqref{eq:desired-ineq-12-claim1} is 
        \begin{align}
    \left|\operatorname{det}\left(\begin{array}{ccc}1 & 1 & 1 \\ 0 & r & 1 \\ 0 & 1 & s\end{array}\right)\right|^p   
            &= (1-rs)^p.
        \end{align}
        For the right-hand side of \eqref{eq:desired-ineq-12-claim1}, a direct computation gives 
        \[
             \sum_{i \neq j} A_i B_j = A_1 (1-s)^p + A_2 s^p + A_3,
        \]
        and $A_1+A_3 = 1, A_2 +A_3 = (1-r)^p$. When $r =1$, the inequality is an equality.
        For fixed $r <1$, we define 
        \[
            F(s) = A_1 (1-s)^p + A_2 s^p + A_3 - (1-rs)^p.
        \]
        It suffices to prove that $F(s) \geq 0$. We first note that $F(0) = F(1) = 0$.
        We compute for $s \in (0,1)$ that
        \begin{align}
            F'(s) &= -pA_1 (1-s)^{p-1} +pA_2s^{p-1} + pr(1-rs)^{p-1} 
            \\
            &= p(1-s)^{p-1} \left( -A_1 +A_2\left(\frac{s}{1-s}\right)^{p-1}+ r \left(\frac{1-rs}{1-s}\right)^{p-1} \right) .
        \end{align}
        Notice that the sign of $F'(s)$ is the same as the sign of 
        $$
            H(s) := -A_1 +A_2\left(\frac{s}{1-s}\right)^{p-1}+ r \left(\frac{1-rs}{1-s}\right)^{p-1} = -A_1 +A_2\left(\frac{1-s}{s}\right)^{1-p}+ r \left(\frac{1-s}{1-rs}\right)^{1-p}. 
        $$ 
        Observe that
        \[
            s \longmapsto \frac{1-s}{s} \quad \text{and} \quad s \longmapsto \frac{1-s}{1-r s}
        \]
        are non-increasing functions on $(0,1)$. Since $p <1$, we get that $H$ is a non-increasing function. Consequently, $H$ can change sign at most once, and only from nonnegative to non-positive. Since $F^{\prime}$ has the same sign as $H$, the function $F$ is first non-decreasing and then non-increasing, where either interval may be empty. Therefore, the minimum of $F$ on $[0,1]$ is attained at an endpoint. Hence, for every $s \in[0,1]$,
$$
F(s) \geq \min \{F(0), F(1)\}=0.
$$
        \item[\bf Case 3:]
        Suppose that $b_1 = 1, b_2 =s ,b_3 = 0$. The left-hand side of \eqref{eq:desired-ineq-12-claim1} is 
        \begin{align}
            \left|\operatorname{det}\left(\begin{array}{ccc}1 & 1 & 1 \\ 0 & r & 1 \\ 1 & s & 0\end{array}\right)\right|^p   
            &= |r-(1-s)|^p.
        \end{align}
        For the right-hand side of \eqref{eq:desired-ineq-12-claim1}, a direct computation gives 
        \[
             \sum_{i \neq j} A_i B_j = A_1 s^p + A_2 + A_3(1-s)^p .
        \]
        This is exactly the expression from the first case, replacing $s$ by $1-s$.
        \item[\bf Case 4:] 
        Suppose that $b_1 = 1, b_2 =0 ,b_3 = s$. The left-hand side of \eqref{eq:desired-ineq-12-claim1} is 
        \begin{align}
            \left|\operatorname{det}\left(\begin{array}{ccc}1 & 1 & 1 \\ 0 & r & 1 \\ 1 & 0 & s\end{array}\right)\right|^p   
            &= (1-r(1-s))^p.
        \end{align}
        For the right-hand side of \eqref{eq:desired-ineq-12-claim1}, a direct computation gives 
        \[
             \sum_{i \neq j} A_i B_j = A_1 s^p + A_2(1-s)^p +A_3.
        \]
        This is exactly the expression from the second case, replacing $s$ by $1-s$.
    \end{enumerate}
    The remaining two cases are equivalent to the second and fourth cases by interchanging the indices 1 and 3. Hence, we have established the claim.

    Now, we show that the right-hand side of \eqref{eq:desired-ineq-12} is the same as the right-hand side of the claim \eqref{eq:desired-ineq-12-claim1}.
    We recall the representation in Lemma \ref{lem:representation-p-01}, for any $ a,b \in \R$,
    \begin{equation} \label{eq:representation-p-01}
        |a-b|^p=\int \left|{1}_{[s, s+t]}(a)-{1}_{[s, s+t]}(b)\right| d \mu(s,t).
    \end{equation}
    To lighten the notation, we denote 
    \begin{center}
        $c_i (s,t) = {1}_{[s, s+t]}(a_i)$ \quad and \quad  $d_i(s,t) = {1}_{[s, s+t]}(b_i)$\qquad for $ 1\leq i \leq 3$.
    \end{center} 
    Using \eqref{eq:representation-p-01}, we obtain that 
    \begin{align}
        A_i = \frac{|a_i - a_j|^p + |a_i - a_k|^p - |a_k - a_j|^p}{2} =
        \int C_i(s,t)  d \mu(s,t), \qquad  1\leq i \leq 3,
    \end{align}
    where 
    \[
        C_i(s,t) :=  \frac{|c_i (s,t)- c_j(s,t)| + |c_i (s,t)- c_k(s,t)| - |c_j(s,t) - c_k(s,t)|}{2}.
    \]
    Here, $j,k$ are two distinct  indices from $\{1,2,3\} \setminus \{i\}$.
    Similarly, we have
    \begin{align}
        B_i =\int D_i(\sigma, \tau)  d \mu(\sigma,\tau), \qquad  1\leq i \leq 3,
    \end{align}
    where $D_i$ is defined in the same way as $C_i$, with $c_l$ replaced by $d_l$ for each $1 \leq l\leq 3$.
    Hence,
    \begin{equation} \label{eq:claim-001}
        \sum_{i \neq j} A_i B_j  =  \iint \sum_{i \neq j} C_i (s, t) D_j(\sigma, \tau) \, d \mu(s,t)   d \mu(\sigma,\tau).
    \end{equation}
    It follows from Lemma \ref{lem:desired-ineq-12-01} that for any $(s,t) ,(\sigma,\tau) \in \R\times (0,\infty)$,
    \begin{equation} \label{eq:claim-p-01}
        \sum_{i \neq j} C_i(s, t) D_j(\sigma, \tau)
        = 
        \left|\operatorname{det}\left(\begin{array}{ccc}1 & 1 & 1 \\ c_1(s,t) & c_2(s,t) & c_3(s,t) \\ d_1(\sigma, \tau) & d_2(\sigma, \tau) & d_3(\sigma, \tau)\end{array}\right)\right|.
    \end{equation}
    Inserting it into \eqref{eq:claim-001}, we complete the proof. \qedhere
\end{proof}

\begin{proof}[Proof of the positive part of Theorem \ref{thm:determinant-ineq-weak}]
    By applying an orthogonal transformation, we may assume that $v = e_3$. 
    We prove that adding one vector does not decrease the ratio, that is,
    \begin{equation} \label{eq:main-goal}
        \frac{\displaystyle\sum_{1 \leq i< j < k \leq N}\left|\operatorname{det}\left(u_i,u_j,u_k\right)\right|^p}{\displaystyle\sum_{1 \leq i< j  \leq N}\left|\operatorname{det}\left(u_i,u_j,e_3\right)\right|^p} 
        \leq 
        \frac{\displaystyle\sum_{1 \leq i< j < k \leq N+1}\left|\operatorname{det}\left(u_i,u_j,u_k\right)\right|^p}{\displaystyle\sum_{1 \leq i< j  \leq N+1}\left|\operatorname{det}\left(u_i,u_j,e_3\right)\right|^p}.
    \end{equation}
    If $u_{N+1} $ is parallel to $e_3$, then the denominator of the right-hand side is the same as the denominator on the left-hand side, while the numerator on the right-hand side contains additional terms. Hence, \eqref{eq:main-goal} follows. We therefore assume that $u_{N+1}$ and $e_3$ are linearly independent. Applying an invertible linear map that sends  $u_{N+1}$ to $e_2$ and fixes $e_3 $ and using cross multiplication, it suffices to prove 
    \begin{align}
        \label{eq:w-det-p-01-pf}
         &\left(\sum_{1\leq i \leq N}\left|\operatorname{det}\left(u_i, e_2, e_3\right)\right|^p\right)\left(\sum_{1\leq i<j<k\leq N}\left|\operatorname{det}\left(u_i, u_j, u_k\right)\right|^p\right)
         \\
         &\qquad \leq
         \left(\sum_{1\leq i<j\leq N}\left|\operatorname{det}\left(u_i, u_j, e_3\right)\right|^p\right)\left(\sum_{1\leq i<j \leq N}\left|\operatorname{det}\left(u_i, u_j, e_2\right)\right|^p\right) .
    \end{align}
    In what follows, all indices in the summation are assumed to range from $1$ to $N$.
    Note that the case $p =1$ was proved in \cite[Theorem 5.1]{FMMZ-24}. Now, we assume that $p <1$.
    For any $1\leq i \leq N$, write $ u_i = (x_i,y_i,z_i)$. Let us assume first that $x_i \neq 0$ for all $i$. Set
    \begin{center}
        $a_i = \frac{y_i}{x_i}$, \quad  $b_i = \frac{z_i}{x_i}$,
        \quad and \quad $w_i = |x_i|^p > 0$.
    \end{center}
    Hence, \eqref{eq:w-det-p-01-pf} reduces to 
    \begin{equation}
    \begin{split}
    &\left(\sum_{i}w_i\right) 
    \left(\sum_{i<j<k} w_iw_jw_k \left|\operatorname{det}\left(\begin{array}{ccc} 1 & 1 & 1 \\ a_i & a_j & a_k \\ b_i & b_j & b_k\end{array}\right)\right|^p\right)
    \\
    &\qquad\leq
    \left(\sum_{ i<j} w_iw_j \left|a_i-a_j\right|^p\right) \left(\sum_{i<j} w_iw_j \left|b_i-b_j\right|^p\right).
    \end{split} \label{eq:claim-002}
    \end{equation}
    Define a measure $\mu$ on $\R \times (0,\infty)$ by 
    $$
        d \mu (s,t) = \frac{p(1-p)}{2} d s \, t^{p-2} d t.
    $$ 
    Using Lemma \ref{lem:desired-ineq-12}, we obtain that
     \begin{equation} 
        \begin{split}
        &\left|\operatorname{det}\left(\begin{array}{ccc}1 & 1 & 1 \\ a_i & a_j & a_k \\ b_i & b_j & b_k\end{array}\right)\right|^p  
        \\
        &\qquad \leq
        \iint\left|\operatorname{det}\left(\begin{array}{ccc}1 & 1 & 1 \\ 1_{[s,s+t]}(a_i) & 1_{[s,s+t]}(a_j) & 1_{[s,s+t]}(a_k) \\ 1_{[\sigma,\sigma+\tau]}(b_i) & 1_{[\sigma,\sigma+\tau]}(b_j) & 1_{[\sigma,\sigma+\tau]}(b_k)
        \end{array}\right)\right| d \mu(s, t) d \mu(\sigma, \tau).
        \end{split}
    \end{equation}
    Inserting into the left-hand side of \eqref{eq:claim-002}, we simplify to get that
    \begin{align}
        &\left(\sum_{i }w_i\right)  \sum_{i<j<k} w_iw_jw_k \left|\operatorname{det}\left(\begin{array}{ccc} 1 & 1 & 1 \\ a_i & a_j & a_k \\ b_i & b_j & b_k\end{array}\right)\right|^p
        \\
        &\qquad \leq
        \iint
        \left(\sum_{i}w_i\right)  \sum_{i<j<k}  \left|\operatorname{det}\left(\begin{array}{ccc}w_i & w_j & w_k \\ c_i  & c_j  & c_k  \\ d_i & d_j & d_k
        \end{array}\right)\right| d \mu(s, t) d \mu(\sigma, \tau),
    \end{align}
    where for any $i$, $c_i = w_i \cdot 1_{[s,s+t]}(a_i) $ and $d_i = w_i \cdot 1_{[\sigma,\sigma+\tau]}(b_i)$. Applying \eqref{eq:w-det-p-01-pf} for $p =1$  to the integrand gives
    \begin{align}
        &\left(\sum_{i }w_i\right)  \sum_{i<j<k} w_iw_jw_k \left|\operatorname{det}\left(\begin{array}{ccc} 1 & 1 & 1 \\ a_i & a_j & a_k \\ b_i & b_j & b_k\end{array}\right)\right|^p
        \\
        &\qquad \leq \iint \left(\sum_{i<j}\left|\operatorname{det}\left(\begin{array}{cc}w_i & w_j  \\ c_i & c_j 
        \end{array}\right)\right|\right)\left(\sum_{i<j }\left|\operatorname{det}\left(\begin{array}{cc}w_i & w_j \\  d_i & d_j
        \end{array}\right)\right|\right) d \mu(s, t) d \mu(\sigma, \tau)
        \\
        &\qquad = \left(\sum_{i<j} \int  \left|\operatorname{det}\left(\begin{array}{cc}w_i & w_j  \\ c_i & c_j 
        \end{array}\right)\right| d \mu(s, t)\right)\left(\sum_{i<j }\int \left|\operatorname{det}\left(\begin{array}{cc}w_i & w_j \\  d_i & d_j
        \end{array}\right)\right|d \mu(\sigma, \tau)\right) ,
    \end{align}
    where we used Fubini's theorem in the last step. Using Lemma  \ref{lem:representation-p-01}, we get that 
    \begin{align}
        \sum_{i<j}  \int  \left|\operatorname{det}\left(\begin{array}{cc}w_i & w_j  \\ c_i & c_j 
        \end{array}\right)\right| d \mu(s, t)
        &= \sum_{i<j} w_iw_j \int \left|1_{[s,s+t]}(a_i) -1_{[s,s+t]}(a_j) \right| d \mu(s, t) 
        \\
        &=\sum_{i<j}  w_iw_j | a_i -a_j|^p.
    \end{align}
    The same argument applies to the second factor. Hence, we obtain \eqref{eq:claim-002}, which is equivalent to \eqref{eq:w-det-p-01-pf}.

    The general case follows by continuity. Indeed, replace
\(u_i=(x_i,y_i,z_i)\) by
\[
    u_i^{(\varepsilon)}=(x_i+\varepsilon,y_i,z_i),
\]
where \(\varepsilon\to0\) through values for which
\(x_i+\varepsilon\ne0\) for every \(i\). Applying the preceding
inequality to \(u_i^{(\varepsilon)}\) and passing to the limit proves
the general case.
\end{proof}

\appendix
\section{The duality-map estimate}\label{app:duality-map}
For the reader's convenience, we recall the statement of
Lemma~\ref{lem:mazur-map-section}. Let $1<p\leq2$, and let $q$ be the
H\"older conjugate of $p$, so that
$
\frac{1}{p}+\frac{1}{q}=1.
$
For a nonzero vector $u\in\R^m$, define
$$
J(u)=
\frac{(|u_1|^{p-2}u_1,\ldots,|u_m|^{p-2}u_m)}
{\|u\|_p^{p-1}}.
$$
Then, for all nonzero $u,v\in\R^m$,
$$
\|J(u)-J(v)\|_q
\leq
2\left(
\frac{\|u-v\|_p}{\|u\|_p+\|v\|_p}
\right)^{p-1}.
$$
\begin{proof}[Proof of Lemma~\ref{lem:mazur-map-section}] 
Set $\phi(t)=|t|^{p-2}t,$ for $ t\in\R\setminus\{0\},$ and $\phi(0)=0.$ We first prove the following scalar inequality. If
$\lambda,\mu\geq0$, $\lambda+\mu=1$, and $s,t\in\R$, then
\begin{equation}
\label{eq:scalar-mazur-map-section-claim}
    |\phi(s)-\phi(t)|^q
    \leq
    2^q|\lambda s-\mu t|^p
    +
    2^{q-1}(\mu-\lambda)
    \bigl(|s|^p-|t|^p\bigr).
\end{equation}
The inequality is invariant under the simultaneous interchange $(s,\lambda)\leftrightarrow(t,\mu)$. Thus, we may assume that $|s| \geq |t|$. By continuity, it suffices to consider the case $
|s|\geq |t|>0$. The inequality is also unchanged if $(s,t)$ is replaced by $(-s,-t)$, so we may assume that $t > 0$. Since the inequality is homogeneous of degree $p$, we divide by $t$ and reduce to the case $t=1$. Hence, by substituting $\mu = 1-\lambda$, it remains to prove for $s \in \R$ with $|s| \geq 1$, and $\lambda \in [0,1]$,
\begin{equation}
\label{eq:scalar-mazur-map-section}
    |\phi(s)-1|^q
    \leq
    2^q|\lambda (s+1)-1 |^p
    +
    2^{q-1}(1-2\lambda)
    \bigl(|s|^p-1\bigr).
\end{equation}
\begin{enumerate}[wide =0pt, labelwidth= 1.1cm]
    \item[\bf Case 1:] Suppose that $s \geq 1$. If $s = 1$, then inequality \eqref{eq:scalar-mazur-map-section} is immediate since the left-hand side is zero. Hence, we assume that $s > 1$. Fix $s$. We minimize the right-hand side of \eqref{eq:scalar-mazur-map-section} with respect to $\lambda$. Define a function $F$ on $[0,1]$ by 
    \[
        F(\lambda) = 2^q|\lambda(s+1)-1|^p+2^{q-1}(1-2 \lambda)\left(s^p-1\right).
    \]
    Differentiating $F$ we obtain 
    \[
        F^{\prime}(\lambda)=2^q p(s+1) \operatorname{sgn}(\lambda(s+1)-1)|\lambda(s+1)-1|^{p-1}-2^q\left(s^p-1\right).
    \]
    The critical-point equation is 
    $$
        p(s+1)\operatorname{sgn}(\lambda(s+1)-1)|\lambda(s+1)-1|^{p-1}=s^p-1 .
    $$
    Since $ s > 1$, the right-hand side is positive. Therefore,  any solution must satisfy $\lambda > \frac{1}{s+1}$. Solving the equation, the critical point is
    \[
        \lambda_0 =\frac{1+\left(\frac{s^p-1}{p(s+1)}\right)^{1 /(p-1)}}{s+1}.
    \]
    Since $1< p \leq 2$, the function $x \mapsto x^{p-1}$ is concave. Using Jensen's inequality, we have 
    \begin{equation}
    \label{eq:Jen-with-s}
        \frac{s^p-1}{p(s-1)} =\frac{1}{s-1}\int_1^s t^{p-1} dt \leq \left(\frac{1}{s-1} \int_1^s t d t \right)^{p-1} = \left(\frac{s+1}{2}\right)^{p-1}.
    \end{equation}
    Hence,
    \begin{equation} 
    \frac{s^p-1}{p(s+1)}
\leq
    \frac{s-1}{s+1}
    \left(\frac{s+1}{2}\right)^{p-1}
    \leq
    \left(\frac{s-1}{2}\right)^{p-1}. \label{eq:bound-of-F}
    \end{equation}
    Therefore, $\lambda_0 \leq \frac{1}{2}$. Since $F$ is strictly convex, $\lambda_0$ is the only minimizer of $F$. 
    Substituting $\lambda_0$ into $F$, we obtain 
    \begin{align}
        F\left(\lambda_0\right)&=2^q  \left(\frac{s^p-1}{p(s+1)}\right)^{p /(p-1)} 
         +2^{q-1}\left(1-\frac{2+2\left(\frac{s^p-1}{p(s+1)}\right)^{1 /(p-1)}}{s+1}\right)\left(s^p-1\right)
         \\ & 
         =\frac{2^{q-1}\left(s^p-1\right)}{s+1}\left[s-1-\frac{2(p-1)}{p}\left(\frac{s^p-1}{p(s+1)}\right)^{1 /(p-1)}\right] .
    \end{align}
    Using \eqref{eq:bound-of-F}, we obtain
    \[
        F(\lambda_0) \geq \frac{2^{q-1}\left(s^p-1\right)}{s+1} \cdot \frac{s-1}{p}.
    \]
    To verify \eqref{eq:scalar-mazur-map-section}, it is enough to prove
    \begin{equation} \label{eq:to-prove-final-s}
        |s^{p-1} -1|^q \leq \frac{2^{q-1}\left(s^p-1\right)(s-1)}{p(s+1)}, \qquad s> 1.
    \end{equation}
    Define $G$ on $[1,\infty)$ by
    \[
        G(s)=
        \frac{
            2^{q-1}(s^p-1)(s-1)
        }{
            p(s+1)(s^{p-1}-1)^q
        }.
    \]
    Taking logarithms on both sides and then differentiating, we obtain
    \[
        \frac{G'(s)}{G(s)}
   =
    \frac{p s^{p-1}}{s^p-1}
    +
    \frac{1}{s-1}
    -
    \frac{1}{s+1}
    -
    \frac{p s^{p-2}}{s^{p-1}-1}
=
    -\frac{
        p s^{p-2}(s-1)
    }{
        (s^p-1)(s^{p-1}-1)
    }
    +
    \frac{2}{s^2-1}.
    \]
    Claim that $G' (s) \leq 0$. Equivalently, 
    \[
        2 \left( \frac{(s^p -1)(s^{p-1} -1)}{p(s-1)^2} \right) 
        \leq s^{p-2} (s+1).
    \]
    Using \eqref{eq:Jen-with-s} and the Hermite--Hadamard inequality with the convex function $t \mapsto t^{p-2}$ on $ [1,s]$, we obtain
    \begin{align}
        2 \left( \frac{(s^p -1)(s^{p-1} -1)}{p(s-1)^2} \right) &=2 (p-1)\left( \frac{s^p -1}{p(s-1)} \right) \left( \frac{s^{p-1} -1}{(p-1)(s-1)} \right)
        \\
        &\leq 
        2(p-1) \left(\frac{s+1}{2}\right)^{p-1}\left( \frac{1}{s-1} \int_{1}^s t^{p-2} dt \right)
        \\
        &\leq 
        2(p-1) \left(\frac{s+1}{2}\right)^{p-1} \left( \frac{s^{p-2}+1}{2} \right)
        \\
        &\leq 2^{1-p}s^{p-2}(s+1)^{p-1} (1 + s^{2-p}). \label{eq:final-first-case}
    \end{align}
    Using H\"older's inequality, we obtain
    \[
        1 + s^{2-p} \leq  2^{p-1}(1+s)^{2-p}.
    \]
    Plugging it into \eqref{eq:final-first-case}, we verify that $G'(s) \leq 0$ for $s >1$. Moreover, since $ q \geq 2$ and $p \leq 2$, we get
\[
    \lim_{s\to\infty}G(s)
    =
    \frac{2^{q-1}}{p} \geq 1.
\]
Thus, $G(s)\geq1$ for $s >1$, proving \eqref{eq:to-prove-final-s} and hence
\eqref{eq:scalar-mazur-map-section} in the case $s\geq1$.
\item[\bf Case 2:] Suppose that $s \leq -1$. If $s = -1$, then both sides of inequality \eqref{eq:scalar-mazur-map-section} are equal. Thus, we only need to prove when $s < -1$. To avoid confusion caused by the sign of $s$, we write $s=-r$ where $ r > 1$. Thus, we need to prove
\begin{equation}
\label{eq:scalar-mazur-map-section-2}
    (r^{p-1}+1)^q
    \leq
    2^q(1+ \lambda (r-1))^p
    +
    2^{q-1}(1-2\lambda)
    \bigl(r^p-1\bigr).
\end{equation}
Fix $r$, and define $F$ on $[0,1]$ by
$$
F(\lambda)=2^q(1+\lambda(r-1))^p+2^{q-1}(1-2 \lambda)\left(r^p-1\right) .
$$
Differentiating, we obtain
$$
F^{\prime}(\lambda)=2^q p(r-1)(1+\lambda(r-1))^{p-1}-2^q\left(r^p-1\right) .
$$
The critical-point equation is
$
p(r-1)(1+\lambda(r-1))^{p-1}=r^p-1 .
$
Therefore, the critical point is
$$
\lambda_0=\frac{\left(\frac{r^p-1}{p(r-1)}\right)^{1 /(p-1)}-1}{r-1} .
$$
Since $t^{p-1} \geq 1$ on $[1, r]$,
$$
\frac{r^p-1}{p(r-1)}=\frac{1}{r-1} \int_1^r t^{p-1} d t \geq 1.
$$
Hence,
$
\lambda_0 \geq 0 .
$
Moreover, since $1<p \leq 2$, the function $t \mapsto t^{p-1}$ is concave. By Jensen's inequality,
\begin{equation} \label{eq:second-case-claim-001}
\frac{r^p-1}{p(r-1)} \leq \left(\frac{r+1}{2}\right)^{p-1}.
\end{equation}
Thus, $\lambda_0 \leq \frac{1}{2}$. Since $F$ is strictly convex, $\lambda_0$ is its unique minimizer. Substituting $\lambda_0$ into $F$, we obtain
\[
    F(\lambda_0 ) = 2^{q-1} \frac{r^p-1}{p(r-1)}\left[p(r+1)-2(p-1)\left(\frac{r^p-1}{p(r-1)}\right)^{1 /(p-1)}\right].
\]
Using \eqref{eq:second-case-claim-001}, we obtain
$$
F\left(\lambda_0\right) \geq 2^{q-1} \frac{r^p-1}{p(r-1)}(r+1) .
$$
We use the Hermite--Hadamard inequality. Since $t \mapsto t^{p-1}$ is concave,
$$
\frac{r^p-1}{p(r-1)}=\frac{1}{r-1} \int_1^r t^{p-1} d t \geq \frac{1+r^{p-1}}{2}.
$$
It follows that
$
F\left(\lambda_0\right) \geq 2^{q-2}\left(1+r^{p-1}\right)(r+1) .
$
To prove \eqref{eq:scalar-mazur-map-section-2}, it remains to check
\begin{equation} \label{eq:final-case-2}
\left(r^{p-1}+1\right)^q \leq 2^{q-2}\left(1+r^{p-1}\right)(r+1).
\end{equation}
By H\"older's inequality, we have
$
1+r^{p-1} \leq 2^{2-p}(1+r)^{p-1} .
$
Thus,
$$
    \left(1+r^{p-1}\right)^{q-1}  \leq 2^{(2-p)(q-1)}(r+1)^{(p-1)(q-1)}  =2^{q-2}(r+1) ,
    $$
    proving \eqref{eq:final-case-2}, hence \eqref{eq:scalar-mazur-map-section-2}.
\end{enumerate}

    Combining these two cases proves \eqref{eq:scalar-mazur-map-section-claim}. 
    Write
    \[
    J(u)=\left(\phi\left(\frac{u_1}{\|u\|_p}\right), \ldots, \phi\left(\frac{u_m}{\|u\|_p}\right)\right),
    \quad \text{and}\quad
    J(v)=\left(\phi\left(\frac{v_1}{\|v\|_p}\right), \ldots, \phi\left(\frac{v_m}{\|v\|_p}\right)\right).
\]
Applying \eqref{eq:scalar-mazur-map-section-claim} with
$s=\frac{u_i}{\|u\|_p}$, $t=\frac{v_i}{\|v\|_p}$ and $\lambda = \frac{\|u\|_p}{\|u\|_p + \|v\|_p}$, we obtain
\[
\left|\phi\left(\frac{u_i}{\|u\|_p}\right)-\phi\left(\frac{v_i}{\|v\|_p}\right)\right|^q
    \leq
    2^q\frac{|u_i -v_i|^p}{(\|u\|_p + \|v\|_p)^p}
    +
    2^{q-1}\left(\frac{\|v\|_p - \|u\|_p}{\|u\|_p + \|v\|_p}\right)
    \left(\frac{\left|u_i\right|^p}{\|u\|_p^p}-\frac{\left|v_i\right|^p}{\|v\|_p^p}\right).
\]
Summing over $i$, we have
\[
    \|J(u)-J(v)\|_q^q
    \leq
    2^q
    \left(
        \frac{\|u-v\|_p}
        {\|u\|_p+\|v\|_p}
    \right)^p,
\]
where the last term vanishes. Therefore, raising both sides to the power $1 / q$ completes the proof.
\end{proof}

\section*{Acknowledgments and AI assistance disclosure}

The authors thank LAMA at Université Gustave Eiffel and the Institut Mathématique de Jussieu for their hospitality. 

A. M. and A. Z. are supported in part by the U.S. National Science Foundation Grant DMS-2247771 and the United States--Israel Binational Science Foundation (BSF) Grant 2018115.

A. M. was supported by the Chateaubriand Fellowship of the Office for Science \& Technology of the Embassy of France in the United States.

During the preparation of this paper, OpenAI's GPT-5.6 Sol was used as an auxiliary tool to explore examples, test determinant computations, and assist with preliminary proof development. The mathematical arguments, final statements, and computations presented here are those of the authors, who have verified them and take full responsibility for the content of the paper.


\bigskip
\noindent Matthieu Fradelizi
\\
Univ Gustave Eiffel, Univ Paris Est Creteil, CNRS, LAMA UMR8050, F-77447 Marne-la-Vall\'ee, France.
\\
E-mail address: matthieu.fradelizi@univ-eiffel.fr
\vspace{2mm}
\\
\noindent Auttawich Manui 
\\
Department of Mathematics, Syracuse University, Syracuse, NY 13244, USA
\\
E-mail address: amanui@syr.edu
\vspace{2mm}
\\
\noindent Cheikh Saliou Ndiaye
\\
Univ Gustave Eiffel, Univ Paris Est Creteil, CNRS, LAMA UMR8050, F-77447 Marne-la-Vall\'ee, France.
\\
E-mail address: cheikh-saliou.ndiaye@univ-eiffel.fr
\vspace{2mm}
\\
\noindent Artem Zvavitch
\\
Department of Mathematical Sciences, Kent State University, Kent, OH 44242, USA. 
\\ E-mail address: zvavitch@math.kent.edu

\end{document}